\documentclass[a4paper, 10pt]{amsart}

\usepackage{amsmath, latexsym,  amssymb, amscd}
\usepackage[all]{xy}
\usepackage[british]{babel}
\usepackage{url}
\usepackage{graphicx}

\usepackage{enumitem}

\usepackage{hyperref}
\usepackage[foot]{amsaddr}

\usepackage{algorithm,algorithmic}

\renewcommand{\epsilon}{\varepsilon}
\renewcommand{\phi}{\varphi}

\renewcommand{\emptyset}{\varnothing}

\DeclareMathOperator{\End}{End}

\newcommand{\abs}[1]{\left| #1 \right|}
\newcommand{\norm}[1]{\left\lVert #1 \right\rVert}
\newcommand{\scal}[2]{\left\langle #1 , #2 \right\rangle}
\renewcommand{\Re}{\operatorname{Re}}
\renewcommand{\Im}{\operatorname{Im}}

\newcommand{\rank}{{\rm rank\,}}

\renewcommand{\P}{\mathbb{P}}

\newcommand{\sotto}{B}

\newcommand{\rend}{\End(E)}

\newcommand{\qe}{\mathbb{Q}}
\newcommand{\C}{\mathbb{C}}
\newcommand{\R}{\mathbb{R}}
\newcommand{\Z}{\mathbb{Z}}
\newcommand{\Oo}{{\mathcal{O}}}

\newcommand{\Q}{\mathbb{Q}}

\newcommand{\Ci}{\mathcal{C}}
\newcommand{\Di}{\mathcal{D}}
\newcommand{\cquattro}{C}
\newcommand{\ccinque}{C}
\newcommand{\csei}{C_0}
\newcommand{\cdodici}{c_{1}}
\newcommand{\ctredici}{c}
\newcommand{\cdiciassette}{c_{2}}
\newcommand{\cuno}{c_3}
\newcommand{\cdue}{c_4}

\newcommand{\Cuno}{C_1}
\newcommand{\Cdue}{C_2}
\newcommand{\Ctre}{C_{4}}
\newcommand{\Cquattro}{C_{3}}
\newcommand{\Ccinque}{C_5}
\newcommand{\Csei}{C_6}
\newcommand{\Csette}{C_7}
\newcommand{\Cotto}{C_8}
\newcommand{\Duno}{D_1}
\newcommand{\Ddue}{D_2}
\newcommand{\Dtre}{D_{3}}

\newcommand{\Dcinque}{D_{4}}
\newcommand{\Dsei}{D_{5}}
\newcommand{\Dsette}{D_{6}}

\newcommand{\csette}{c_{6}}
\newcommand{\cotto}{c_{7}}
\newcommand{\cnove}{c_{8}}

\newcommand{\cquattordici}{c_{12}}
\newcommand{\cquindici}{c_{11}}

\newtheorem{thm}{Theorem}[section]

\newtheorem{propo}[thm]{Proposition}
\newtheorem{lem}[thm]{Lemma}
\newtheorem{cor}[thm]{Corollary}
\newtheorem{D}[thm]{Definition}

\newtheorem{remark}[thm]{Remark}

\usepackage{hyperref}

\title[]{Explicit height bounds for $K$-rational points on transverse curves in powers of elliptic curves}

\author{F. Veneziano and E. Viada}

\subjclass[2010]{Primary 11G50, Secondary 14G40}
\begin{document}

%\keywords{Subvarieties of products of elliptic curves, Finiteness of torsion anomalous intersections, Diophantine approximation}

\begin{abstract} 
Let  $\Ci$ be an algebraic  curve embedded transversally in a power  $E^N$ of an elliptic curve $E$. In this article we produce a good explicit bound for the height of all the algebraic points on $\Ci$ contained in the union of all proper algebraic subgroups of $E^N$. The method gives a totally explicit version of the Manin-Dam'janenko Theorem in the elliptic case and it is a generalisation of previous results only proved when $E$ does not have Complex Multiplication.\end{abstract}
\maketitle

\section{Introduction}

The Mordell Conjecture, proved by Faltings  (\cite {FaltingsTeo}), states  that an algebraic  curve of genus at least $2$ defined over a number field $k$ has only finitely many $k$-rational points. 
As it is well known, Faltings'  proof is not effective  in the sense that it does not give any information on how these points could be determined. 
The main known effective methods for finding rational points on algebraic curves are the Chabauty-Coleman method  and the  Manin-Dem'janenko method.

%The main known effective methods for finding rational points on algebraic curves are the Chabauty-
%Coleman method (see \cite{Chab}  and  \cite{Coleman}) and the  Manin-Dem'janenko method (see %\cite{Demj} and \cite{Manin}); see Serre's book \cite[Chapter 5]{serre} for an overview.
%The  Chabauty-Coleman method has been used by several mathematician to determine rational points and it %turned out to be successful for example on some families of hyperelliptic curves with special Jacobian.  The %method by Manin and Dem'janenko applies to curves  defined over a number field $k$ that admit $m$ %different $k$-independent morphisms towards an abelian variety $A$ defined over $k$ with rank of $A(k)$ %smaller than $m$. This method too has been investigated to produce examples of curves of genus $2$ for %which the set of rational points can be computed.

Unfortunately though, these methods do not express the bound for the height of the $k$-rational points as a  formula in terms of the curve. Thus, in the applications, such a dependence must be elaborated on a case-by-case basis with \textit{ad hoc} strategies and this has been carried out successfully only for some special families of curves with small genus, typically $2$ or $3$. (We refer to the introduction of \cite{EsMordell} for an account on the subject). \\

%In this article we investigate the  strength and the limit of an explicit method that we introduced in \cite{EsMordell} and that turned out to be successful for finding the rational points on new families of curves.
In this article we generalise an explicit method that we introduced in \cite{EsMordell} investigating its strength and its limits. More precisely we give a simple formula for the height of the points  in the CM and non-CM case and for rank larger than 1. We are  then successful in finding the $k$-rational (and not only rational) points on  some families of curves, with growing genus. The method  has its roots in  the theory of anomalous intersections introduced by Bombieri, Masser and Zannier.

To discuss these results we first fix the general setting and terminology (see  Section \ref{SezioneAltezze} for more details). By variety we mean an algebraic variety defined over the algebraic numbers embedded in some projective space. For $k$ a number field and $V$ a variety defined over $k$, we denote by $V(k)$ the set of $k$-rational points on $V$. We denote by $E$ an elliptic curve  and for any positive integer $N$ we denote by $E^N$ the cartesian product of $N$ copies of $E$.
We say that a subvariety $V\subset  E^N$  is a \emph{translate},  respectively  a \emph{torsion variety},  if it is a finite  union of translates of algebraic subgroups  of $E^N$ by  points, respectively  by torsion points.

Furthermore, an irreducible variety $V\subset  E^N$  is \emph{transverse}, respectively \emph{weak-transverse}, if  it is not contained in any proper translate, respectively in any proper torsion variety.
We also give the following definition:
 \begin{D}
\label{rank}
 The rank  of a point in $E^N$ is the minimal dimension  of an algebraic subgroup of $E^N$ containing the point.
\end{D}

 In  \cite{EsMordell}  the authors and  Checcoli  gave  a good explicit  bound for the N\'eron-Tate height of the set of points of rank one on algebraic curves of genus at least two in $E^2$ where $E$ is without CM; the sharpness of the bounds allowed explicit examples to be computed.
 
The assumptions in  \cite{EsMordell} represent the easiest setting in this context: points of rank one and $E$  without CM. In that paper we tested the possibility of producing an explicit and even implementable method for finding the rational points on some new families of algebraic curves. 
  In \cite{viaMD} and \cite{viadabasel} Viada extended the previous methods of \cite{ExpTAC} and \cite{EsMordell}, obtaining partial and less sharp results that are however too large to be implemented.
  %{\bf In \cite{viaMD} Viada adapted to the CM case and for higher rank an older and less sharp method introduced in \cite{ExpTAC}; in \cite{viadabasel} she extends the method introduced in \cite{EsMordell} to higher rank for the non-CM case. Her bounds are however too large to be implemented.}
%In \cite{viaMD} and \cite{viadabasel} Viada gives some further bounds in this context, which are however too large to be implemented.

In this article we generalise and improve on all previous results covering all possible cases that our method can solve in the elliptic context. This is a new  explicit version of the Manin-Dem'janenko Theorem in the elliptic setting.  The bound obtained here  is the first implementable generalisation of \cite{EsMordell} for higher rank in the CM and non-CM cases. It can be used to find all $k$-rational points of many new curves as presented in Section \ref{example}.  The independence of our bound on $k$ and on the generators of $E(k)$ is an interesting aspect for the possible further applications.

\bigskip

Let $E$ be an elliptic curve  given in the form
\begin{equation*}\label{uno}y^2=x^3+Ax+B.\end{equation*}
%Let $\Delta$ and $j$ be, respectively, the discriminant and $j$-invariant of $E$. 
Via this equation, we embed $E^N$ into $\P_2^N$ and via the Segre embedding in $\P_{3^N-1}$.  
%, $h_W$ the absolute logarithmic Weil height in the projective space and $h_\infty$ the contribution of the archimedean places to $h_W$. We denote by $h_{\mathcal{W}}(E)$ the Weil height of the point $(1:A^{1/2}:B^{1/3})$. 

The degree of a  curve   $\Ci\subseteq E^N$ is the degree of its image in $\P_{3^N-1}$  and $h_2(\Ci)$ is the normalised height of $\Ci$, which is defined in terms of the Chow form of the ideal of $\Ci$, as done  in  \cite{patrice}.  We denote by $\hat{h}$  the canonical N\'eron-Tate height  on $E^N$.

The following is a weaker version of Theorems \ref{MAINTCM} and \ref{MAINTnonCM}, where the sharper but cumbersome constants $C_i$'s have been replaced by the more compact $D_i$'s. For implementations  it is advisable to use the sharper version.
\begin{thm}\label{thm:insieme}
  Let $E$ be an elliptic curve and let $\Ci$ be a curve transverse  in $E^N$. Then all the points $P$ of rank at most $N-1$ on $\Ci$ have N\'eron-Tate height explicitely bounded as follows:
  \begin{enumerate}
   \item\label{caso_E^N} If $E$ has Complex Multiplication by the field $K$,
   \begin{equation*}
      \hat h(P)\leq \Duno(N,E)\cdot h_2(\Ci)(\deg\Ci)^{N-1} +\Ddue(N,E)(\deg\Ci)^N +\Dtre(N,E).
\end{equation*}
   \item\label{caso_E^N-non-CM}  If $E$ does not have Complex Multiplication,   
   \begin{equation*}
      \hat h(P)\leq \Dcinque(N)\cdot h_2(\Ci)(\deg\Ci)^{N-1} +\Dsei(N,E)(\deg\Ci)^N+\Dsette(N,E).
\end{equation*}
  \end{enumerate}
The constants are given by:
\begin{align*}
 \ctredici(N)&=2N!\left(\frac{N(N-1)^3 3^{N-1}(2N-2)!^2 (N)!}{2^{2N-5}}\right)^{N-1},\\
   \Duno(N,E)&=\ctredici(N)f^N\abs{D_K}^{N^2-\frac{3}{2}N+1}+1,\\
   \Ddue(N,E)&=\ctredici(N)f^N\abs{D_K}^{N^2-\frac{3}{2}N+1}\left(N^2\cquattro(E)+3^{N}\log 2\right),\\
   %\Dquattro(N,E)&=\frac{(N-1)^2 3^{N-1}(N-1)!}{\cuno(N,E)},\\
   \Dtre(N,E)&=(N+1)\cquattro(E)+1,\\
   %\cuno(N,E)&=\frac{N(N-1)^3 3^{N-1}(2N-2)!^2 (N)!\abs{D_K}^{N-1}}{2^{2N-4}},\\
   \Dcinque(N)&=4N!\left(\frac{N^2 (N-1)^2 3^N}{4^{N-3}}N! (N-1)!^4\right)^{N-1},\\
   \Dsei(N,E)&=\Dcinque(N)\left(N^2\cquattro(E)+3^{N}\log 2\right),\\
   \Dsette(N,E)&=(N+1)\cquattro(E)+1,   
 \end{align*}  
 where, in the CM case, $D_K$ is the discriminant of the field of complex multiplication and $f$ is the conductor of $\rend$. For $C(E)$ we take the one defined in Proposition~\ref{confrontoaltezze} or any other bound for the difference between the Weil and the Canonical heights on $E$.
  \end{thm}

 \bigskip
 
 The bound that we obtain for the height of the points is completely explicit and suitable for applications in the context of the Mordell Conjecture; we deduce here a straightforward corollary. To apply the theorem to the $k$-rational points of a curve $\Ci\subset E^N$, it is sufficient to assume that the $\rend$-module generated by the points in $E(k)$ has rank $\le N-1$ as an $\rend$-module. Then clearly all points of $\Ci(k)$ satisfy the hypothesis of the theorem.

This immediately gives the following:
 \begin{cor}
  Let $E$ be an elliptic curve defined over a number field $k$. Let $\Ci$ be a curve transverse  in $E^N$.
  \begin{enumerate}
   \item If $E$ has Complex Multiplication and the rank of $E(k)\otimes_\Z \rend$ as an $\rend$-module is $< N$, then any $k$-rational point $P\in\Ci(k)$ has N\'eron-Tate height bounded as
    \begin{equation*} \hat{h}(P)\leq \Duno(N,E)\cdot h_2(\Ci)(\deg\Ci)^{N-1} +\Ddue(N,E)(\deg\Ci)^N +\Dtre(N,E); \end{equation*}
    \item if $E$ does not have Complex Multiplication and the rank of $E(k)$ is $< N$, then any $k$-rational point $P\in\Ci(k)$ has N\'eron-Tate height bounded as
    \begin{equation*} \hat{h}(P)\leq \Dcinque(N)\cdot h_2(\Ci)(\deg\Ci)^{N-1} +\Dsei(N,E)(\deg\Ci)^N+\Dsette(N,E), \end{equation*}
  \end{enumerate}
where the constants are the same as in Theorem \ref{thm:insieme}.
\end{cor}
 
  In the case that $k$ contains the field of complex multiplication $K$, the set $E(k)$ is itself an $\rend$-module (see e.g. \cite{SilvermanAdvancedTopics}, Theorem 2.2) and its rank as an $\rend$-module is equal to half its rank as an abelian group; we get then the following corollary:
  \begin{cor}
  Let $E$ be an elliptic curve  with  CM  defined over a number field $k$ which contains the field of Complex Multiplication. Assume that $\rank_\Z E(k)< 2 N$. Let $\Ci$ be a curve transverse  in $E^N$. Then any $k$-rational point $P\in\Ci(k)$ has N\'eron-Tate height bounded as
    \begin{equation*} \hat{h}(P)\leq \Duno(N,E)\cdot h_2(\Ci)(\deg\Ci)^{N-1} +\Ddue(N,E)(\deg\Ci)^N +\Dtre(N,E), \end{equation*} 
    with the same constants as in Theorem~\ref{thm:insieme}.
 \end{cor}

 In both corollaries one can obtain sharper constants using the $C_i$'s from Theorems~\ref{MAINTCM} and \ref{MAINTnonCM} instead of the $D_i$'s of Theorem~\ref{thm:insieme}.  
 
\bigskip

{\it Sketch of the proof of Theorem \ref{thm:insieme}}: The proof follows a classical pattern in diophantine approximation.  By definition, a point $P$ of rank at most $N-1$ lies in an algebraic subgroup $B$ of dimension $N-1$. We construct an auxiliary algebraic subgroup $H$ such that the translate $H+P$ approximates $B$ and has height and degree bounded in terms of $h(P)$ and some parameters. We then use the arithmetic B\'ezout Theorem on the intersection of the curve and the auxiliary translate. A good choice of the parameters will produce an upper bound for the height of our starting point. All the inequalities coming into play must be made explicit, moreover the constants must be kept as small as possible for the applications. Central to this purpose is the use of the first Minkowski theorem instead of other easier but less sharp approximations.
This proof generalizes  the proof in \cite{EsMordell} from rank one to higher rank and the proof of \cite{viadabasel}  from $\mathbb{Z}$-lattices to $\mathcal{O}$-lattices, with $\mathcal{O}$ an order in the ring of integers of an imaginary quadratic number field.

This completes all cases of the explicit Mordell Conjecture that can be covered with our method and opens a wide range of examples of curves suitable for determining the $k$-rational points.

For example let us consider the families of curves
 \begin{align*}
  \Ci_n&=\{(x_1,y_1)\times(x_2,y_2)\in E^2\mid x_1^n=y_2\} \text{ and}\\
  \Di_n&=\{(x_1,y_1)\times(x_2,y_2)\in E^2\mid \Phi_n(x_1)=y_2\},
\end{align*}
where $\Phi_n$ is the $n$-th cyclotomic polynomial and $E$ is the elliptic curve defined by the equation $y^2=x^3+2$; then their sets of $\Q(\sqrt{-3})$-rational points is described in the following theorem:
  %\Ci_{n,m}&=\{(x_1,y_1)\times(x_2,y_2)\times(x_3,y_3)\in E^3\mid x_1^n=x_2^m=y_3\}

  \begin{thm}\label{teorema:esempio}
     Let $E,\Ci_n,\Di_n$ be as above. Let $g=(-1:1:1)$ and $O=(0:1:0)$ in $E(\Q)$. Let $\zeta\in\rend$ be a primitive cube root of $1$. Then the sets $\Ci_n(\Q(\sqrt{-3}))$ and $\Di_n(\Q(\sqrt{-3}))$ are described as follows:
     \begin{align*}
	&\Ci_n(\Q(\sqrt{-3}))\setminus\{(O,O)\} = & & \\
	&=\{(a g, b g)\mid a=\pm 1, \pm \zeta, \pm \zeta^2\text{ and }b=1,\zeta,\zeta^2\}	&\text{if }n&\equiv 0\pmod 6\\
	&=\{(a g, b g)\mid a=\pm 1\text{ and }b=-1,-\zeta,-\zeta^2\}	&\text{if }n&\equiv \pm 1\pmod 6\\
	&=\{(a g, b g)\mid a=\pm 1\text{ and }b=1,\zeta,\zeta^2\}	&\text{if }n&\equiv \pm 2\pmod 6\\
	&=\{(a g, b g)\mid a=\pm 1, \pm \zeta, \pm \zeta^2\text{ and }b=-1,-\zeta,-\zeta^2\}	&\text{if }n&\equiv 3\pmod 6,
     \end{align*}
     and
     \begin{align*}
	&\Di_n(\Q(\sqrt{-3}))=   & &  \\
	&=\{(O,O)\}	&\text{ if }n=1,2,\text{ or }n=2p^k\text{ for $k\geq 1$ and $p$ a prime number} &  \\
	&=\{(\pm g,  g)\}\cup\{(O,O)\}	&\text{otherwise.} &  
     \end{align*}
     
  \end{thm}

\section{Preliminaries}\label{SezioneAltezze}

 In this section we introduce the  notations and we recall from \cite{EsMordell} several explicit inequalities between different height functions. 
  We also recall  some basic results in Arithmetic Geometry that play an important role in our proofs: the a\-de\-lic Minkowski Theorem,  the Arithmetic B\'ezout  Theorem and the Zhang  Inequality.\\

In this paper, the word {\it rank} is used with his several different meanings. For clarity, we remember that the rank of a finitely generated abelian group is the number of generators over $\Z$ of its free part; the rank of an $R$-module $M$ for $R$ an integral domain with field of fraction ${\rm frac}{(R)}$ is the dimension of the vector space $M\otimes_R {\rm frac}(R)$; the $k$-rank of an abelian variety $A$ defined over $k$, for $k$ a number field, is the rank of $A(k)$ as an abelian group; the rank of a point on the abelian variety $A=E^N$, for $E$ an elliptic curve was introduced in Definition~\ref{rank}.\\

Let $E$  be an elliptic curve defined over a number field $k$  by a fixed Weierstrass equation
\begin{equation}\label{Weq}
E: y^2=x^3+Ax+B
\end{equation}
with $A$ and $B$  in the ring of  integers  of $k$ (this assumption is not restrictive).  We denote the  discriminant of $E$ by
\[
 \Delta=-16(4A^3+27B^2)\] and the $j$-invariant  by \[j=\frac{-1728(4A)^3}{\Delta}.
\]
 We consider  $E^N$ embedded in $\P_{3^N-1}$ via the following composition map 
\begin{equation}\label{embe}
  E^N \hookrightarrow \P_2^N \hookrightarrow \P_{3^N-1}
\end{equation}
where the first map sends a point $(X_1,\ldots,X_N)$ to $((x_1,y_1),\dotsc,(x_N,y_N))$ (the $(x_i,y_i)$ being the affine coordinates of $X_i$ in the Weierstrass form of $E$) and the second map is the Segre embedding. 
Degrees and  heights are computed with respect to this fixed embedding.

\subsection{Heights of points}

If $P=(P_0:\ldots:P_n)\in \mathbb{P}_n(\overline{\Q})$ is a point in the projective space, then  the absolute logarithmic Weil height of $P$ is defined as
\[h_W(P)=\sum_{v\in \mathcal{M}_K} \frac{[K_v:\Q_v]}{[K:\Q]}\log \max_i \{\abs{P_i}_v\}\]
where $K$ is a field of definition for $P$ and $\mathcal{M}_K$ is its set of places. If $\alpha\in\overline\Q$ then the Weil height of $\alpha$ is defined as $h_W(\alpha)=h_W(1:\alpha)$.

We also define another height which differs  from the Weil height  at the archimedean places:
\begin{equation}\label{Defh2}
h_2(P)=\sum_{v\text{ finite}}\frac{[K_v:\Q_v]}{[K:\Q]}\log \max_i \{\abs{P_i}_v\} +\sum_{v\text{ infinite}}\frac{[K_v:\Q_v]}{[K:\Q]}\log \left(\sum_i \abs{P_i}_v^2\right)^{1/2}.\end{equation}
%If $x$ is an algebraic number, we denote by $h_{\infty}(x)$ the contribution to the Weil height coming from the archimedean places, more precisely:
  %\[
% h_\infty(x)=\sum_{v\text{ infinite}}\frac{[K_v:\Q_v]}{[K:\Q]} \max \{\log \abs{x}_v, 0\}.\]

 For a point $P\in E$ we denote by $\hat h(P)$ its N\'eron-Tate height as defined in \cite{patriceI}  (which is one third of the usual N\'eron-Tate height used also in   \cite{ExpTAC}). 
%(NOTA: perche dare referenza a patrice per il NT height?----perche' usiamo la sua normalizzazione che e' un terzo di quella di Silverman)

If $P=(P_1,\dotsc,P_N)\in E^N$,  
 then for $h$ equal to $h_{W},h_2$ and $\hat{h}$ we define
\begin{equation*}
   h(P)=\sum_{i=1}^N h(P_i).
\end{equation*}
%If $E$ is an elliptic curve given by a Weierstrass equation as in \eqref{Weq}, we define  \begin{equation}\label{hWeierE} h_{\mathcal{W}}(E)=h_W(1:A^{1/2}:B^{1/3})
%\end{equation} to be the absolute logarithmic Weil height of the projective point $(1:A^{1/2}:B^{1/3})$. 

The following proposition directly  follows from \cite[Theorem 1.1]{SilvermanDifferenceHeights} and \cite[Proposition 3.2]{EsMordell}.
\begin{propo}\label{confrontoaltezze}
 
%For $P\in \mathbb{P}_m$, 
%\begin{equation}\label{hh2}
%h_W(P)\leq h_{2}(P)\leq h_W(P)+\log(m+1)/2.
%\end{equation}

%For $P\in E$,  
%\begin{equation}\label{hWhath}-\frac{(h_W(j)+2h_W(\Delta)+2 h_\infty(j))}{24}-0.973 \leq \frac{\hat h(P)}{3}-\frac{h_W(x(P))}{2}\leq \frac{h_W(\Delta)+h_\infty(j)}{12}+1.07.
%\end{equation}

For $P\in E^N$, 
 \begin{equation*}
    \abs{h_2(P)-\hat h (P)}\leq N C(E),
\end{equation*}
where 
$$C(E)=\frac{h_W(\Delta)+3h_W(j)}{4}+\frac{h_W(A)+h_W(B)}{2}+4.$$
%$$ \ccinque(E)=\frac{h_W(\Delta)+h_\infty(j)}{4}+\frac{h_W(j)}{8}+\frac{h_W(A)+h_W(B)}{2}+3.724,$$
%$$ \cquattro(E)=\frac{h_W(\Delta)+h_\infty(j)}{4}+\frac{h_W(A)+h_W(B)}{2}+4.015.$$        

%NON SI USAVA
%Moreover, if $P\in E(\qe)$ then
 %\begin{equation}\label{zimmer}
%-\frac{(3h_{\mathcal W}(E)+7\log2)}{2}\leq  h_{W}(P)-\hat h(P)\leq 3h_{\mathcal W}(E)+6\log 2.
%\end{equation}
%and we can replace  the constants in (\ref{h2hath}) by
%$$\ccinque(E)=\min\left(\frac{\log\abs{\Delta}+h_\infty(j)}{4}+\frac{h_W(j)}{8}+\frac{\log(\abs{A}+\abs{B}+3)}{2}+2.919,3h_{\mathcal W}(E)+4.709\right),$$
%$$ \cquattro(E)=
   %             \min\left(\frac{\log\abs{\Delta}+h_\infty(j)}{4}+\frac{\log(\abs{A}+\abs{B}+3)}{2}+ 3.21, \frac{3h_{\mathcal W}(E)}{2}+2.427\right).$$
\end{propo}
Further details on the relations between the different height functions defined above can be found in \cite[Section 3]{EsMordell}.

\subsection{Heights of varieties} 
For  a subvariety $V\subseteq \P_m$  we denote by $h_2(V)$ the normalised height of $V$ defined  in terms of the Chow form of the ideal of $V$, as done  in  \cite{patrice}. This height extends the height $h_2$ defined for points by formula~\eqref{Defh2} (see~\cite{BGSGreen} equation (3.1.6)).
We also consider  the canonical height $h(V)$, as defined in \cite{patriceI}; when the variety $V$ reduces to a point $P$, then $h(P)=\hat h(P)$  (see~\cite{patriceI}, Proposition 9).

\subsection{The degree of varieties}

The degree of  an irreducible variety $V\subset  \P_m$ is the maximal cardinality of a finite intersection $V\cap L$, with $L$ a linear subspace of dimension equal to the codimension of $V$. The degree is often conveniently computed as an intersection product.

If $X(E,N)$ is the image of $E^N$ in $\mathbb{P}_{3^N-1}$ via the above map, then by \cite{ExpTAC}, Lemma 2.1 we have \begin{equation}\label{degSegre}\deg X(E,N)=3^N N!.\end{equation}

\subsection{ The Arithmetic B\'ezout Theorem}
 The following explicit result is proven by Philippon in  \cite{patrice}, Th\'eor\`eme 3. It describes  the behaviour of the height under intersections. \begin{thm}[Arithmetic B\'ezout theorem]\label{AriBez}
 Let $X$ and $Y$ be irreducible closed subvarieties of $\P_m$ defined over  the algebraic numbers. If $Z_1,\dotsc,Z_g$ are the irreducible components of $X\cap Y$, then
 \[
  \sum_{i=1}^g h_2(Z_i)\leq\deg(X)h_2(Y)+\deg(Y)h_2(X)+\csei(\dim X,\dim Y, m)\deg(X)\deg(Y)
 \]
where
\begin{equation}\label{costaBez}
 \csei(d_1,d_2,m)=\left(\sum_{i=0}^{d_1}\sum_{j=0}^{d_2} \frac{1}{2(i+j+1)}\right)+\left(m-\frac{d_1+d_2}{2}\right)\log2.
\end{equation}
\end{thm}
We note here for simplicity that the constant $\csei(1,N-1,3^N-1)$, which will occur often in our calculations, is bounded above by $3^N\log 2$.

\subsection{The Zhang Inequality}
 In order to state  Zhang's inequality,  we define the essential minimum $ \mu_2(X)$ of an irreducible algebraic subvariety $X\subset \P_m$ as 

\[
\mu_2(X)=\inf\{\theta\in\R\mid\text{the set }\{P\in X\mid  h_2(P)\leq\theta\}\text{ is Zariski dense in }X\}.
\] 

Zhang's inequality relates the essential minimum of a variety with its height and degree. See \cite{Zhang95}, Theorem~5.2 for the original argument and \cite{tiruchirapalli}, Th{\'e}or{\`e}me~3.1 for an account which uses the height functions defined by Philippon.
\begin{thm}[Zhang inequality]\label{Zhang}
Let $X\subset \P_m$ be an irreducible algebraic subvariety. Then 
\begin{equation}\label{zhangh2}
\mu_2(X)\leq\frac{h_2(X)}{\deg X}\leq(1+\dim X)\mu_2(X).
\end{equation}
\end{thm}
We also define a different essential minimum for subvarieties of $E^N$, relative to the height function $\hat h$:
\[
\hat\mu(X)=\inf\{\theta\in\R\mid\text{the set }\{P\in X\mid \hat h(P)\leq\theta\}\text{ is Zariski dense in }X\}.
\]
Using the definitions and a simple limit argument, one sees that Zhang's inequality holds also with $\hat{\mu}$, namely 
\begin{equation}\label{Zhangh^}\hat\mu(X)\leq\frac{h(X)}{\deg X}\leq(1+\dim X)\hat\mu(X).\end{equation}

If $X$ is an irreducible subvariety in $E^N$, using Proposition \ref{confrontoaltezze} we  have
\begin{equation}\label{mu2mu^}
   \abs{\mu_2(X)-\hat\mu(X)}\leq  N\ccinque(E)
\end{equation}
where the  constant $\cquattro(E)$ is defined in Proposition \ref{confrontoaltezze}.

%Finally, using \eqref{mu2mu^}, \eqref{zhangh2} and \eqref{Zhangh^} we  get: \begin{equation}\label{confrontohh2}
% \frac{h_2(X)}{1+\dim X}-N\ccinque(E)\deg X\leq h(X)\leq (1+\dim X)\left(h_2(X)+N\cquattro(E)\deg X\right).
%\end{equation}

\subsection{Complex Multiplication}\label{CM}

We denote by $\rend$ the ring of endomorphisms of $E$. We recall that an elliptic curve $E$ is said to have CM if $\mathrm{End}(E)$ is isomorphic to an order in the ring of integers $\mathcal{O}_K$ of an imaginary quadratic field $K$.

 In this case we write $K=\Q(\sqrt{D})$, for some squarefree negative integer $D$ and we set $\theta=\sqrt{D}$ if $D\not\equiv 1 \pmod 4$ and $\frac{1+\sqrt{D}}{2}$ if $D\equiv 1 \pmod 4$, so that $\mathcal{O}_K=\Z[\theta]$.

We denote by $D_K$ the discriminant of the field $K$. According to the residue class of $D$ modulo 4 we also have that $D_K=D$  if $D\equiv 1 \pmod 4$ and $D_K=4D$ if $D\not\equiv 1 \pmod 4$.

$\mathrm{End}(E)$ is then isomorphic to an order in $\mathcal{O}_K$, which we can write as $\mathbb{Z}+f\mathcal{O}_K$ for a positive integer $f$ called the conductor of the order. We set $\tau=f\theta$, so that $\rend=\Z[\tau]$.

Considering $\mathrm{End}(E)$ as a lattice in $\C$ we can compute the volume of a fundamental parallelogram with respect to the Lebesgue measure, which is then given by $1\cdot \Im \tau= f\Im \theta=\frac{f}{2}\abs{D_K}^\frac{1}{2}$.

\subsection{Algebraic Subgroups}\label{prelim-S} 
The  algebraic subgroups of $E^N$ are well known objects.

Let us identify $E(\C)$ as complex Lie group with the quotient $\C  / {\Lambda_0}$ for a lattice $\Lambda_0 \subset  \C$; then $\C^N$ is identified with the Lie algebra of $E^N(\C )$.
As explained e.g. in \cite[8.9.8]{BG06}, the set of abelian subvarieties of $E^N$ is in natural bijection with the set of complex vector subspaces $W \subset  \C ^N$ for which $W \cap {\Lambda_0}^N$ is a lattice of full rank in $W$; the bijection sends each abelian subvariety to its Lie algebra.

Using this description we can define the \emph{orthogonal complement} of an abelian subvariety $B \subset  E^N$ with Lie algebra $W_B \subset  \C ^N$ as the abelian subvariety  $B^\perp$ with Lie algebra $W^\perp_B$, where $W^\perp_B$ denotes the orthogonal complement of $W_B$ with respect to the canonical hermitian structure of $\C ^N$. Note that $W_B^\perp \cap {\Lambda_0}^N$ is a lattice of full rank in $W_B^\perp$ and its volume can be estimated using the Siegel Lemma over number fields of Bombieri and Vaaler \cite{BomVal}.

Using a different point of view we can think abelian subvarieties as matrices. The Lie algebra of an abelian subvariety $B$ of codimension $r$ is the kernel of a linear map $\varphi_\sotto:\C^N \to \C^r$. We identify $\varphi_\sotto$ with the induced  morphism $\varphi_\sotto:E^N \to E^{r}$. Then  $\ker \varphi_\sotto=\sotto+\tau$ with $\tau$ a torsion set of  cardinality  $T_0$ which is absolutely bounded (\cite{Masserwustholz}).

In turn $\varphi_\sotto$ is identified with a matrix in $\mathrm{Mat}_{r\times N}(\mathrm{End}(E))$ of rank $r$  and, using the geometry of numbers, we can choose the matrix   representing $ \varphi_\sotto$  in such a way that the degree of   $\sotto$ is essentially the product of  the squares of the norms  of the rows of the matrix.

\subsection{Minkowski's theorem for \texorpdfstring{$K$}{K}-lattices}\label{prelim-Mink}
Let $K$ be an imaginary quadratic field like in Section~\ref{CM}. A $K$-lattice $\Lambda$ of rank $r$ is an $\Oo_K$-module of rank $r$ such that $\Lambda\otimes_{\Oo_K}K$ has dimension $r$ over $K$ (\cite{BG06}, Definition C.2.5).

Let $M$ be an $r\times N$ matrix of rank $r$ with coefficients in $\Oo_K$ and let $\Lambda$ be the $K$-lattice generated by the rows of $M$. We write $\det \Lambda=\sqrt{\det(MM^\dagger)}$, where $M^\dagger=\overline{M^t}$ is the transpose of the complex conjugate.

Just as in the real case, we define for $n=1,\dotsc,r$ the $n$-th \emph{successive minimum} of a $K$-lattice $\Lambda$ a
\[
 \lambda_n=\inf\{ t>0\mid tS\text{ contains $n$ linearly independent vectors of $\Lambda$ over $K$}\},
\]
where $S$ is the unit ball in $K^r$.

Then we have the following special case of \cite{BG06}, Theorem C.2.11. See \cite{viaMD} for the deduction of Theorem \ref{AdelMink} from \cite{BG06}, Theorem C.2.11 and C.2.18.

\begin{thm}[Minkowski's second Theorem]\label{AdelMink}
   Let $K,\Lambda,r$ as above. Let $\lambda_1,\dotsc,\lambda_r$ be the successive minima of $\Lambda$. Then 
   \[
   \omega_{2r}(\lambda_1\dotsm \lambda_r)^2\leq 2^r\abs{D_K}^\frac{r}{2}(\det\Lambda)^2,
   \]
   where $D_K$ is the discriminant of $K$ and $\omega_{2r}$ the volume of the unit ball in $\R^{2r}$.
\end{thm}

\section{Bound for the height and degree of translates}\label{section:computations}

\subsection{Computing the degree as an intersection product}

The degree of subvarieties of $\P_2^N$ is often conveniently computed as an intersection product; we show here how to do it for a hypersurface $V$ of $E^N$, and especially for a subgroup.

Let $L$ be the class of a line in the Picard group of $\P_2$, let $\pi_i:\P_2^N\to \P_2$ be the projection on the $i$-th component and $\phi_i:\P_2^N\to\P_2^{N-1}$ the projection which omits the $i$-th coordinate. Set $\ell_i=\pi_i^*(L)$.
The $\ell_i$'s have codimension 1 in $\P_2^N$ and they generate its Chow ring, which is isomorphic as a ring to $\Z[\ell_1,\dotsc,\ell_N]/(\ell_1^{3},\dotsc,\ell_N^{3})$.

The pullback through the Segre embedding of a hyperplane of $\P_{3^N-1}$ is given by $\ell_1+\dotsb+\ell_N$ as can be seen directly from the the equation of a coordinate hyperplane in $\P_{3^N-1}$.
The degree of $V$ (which has dimension $N-1$) is therefore given by the intersection product 
\[\deg V=V.(\ell_1+\dotsb+\ell_N)^{N-1}\]
in the Chow ring of $\P_2^N$.

Define now $d_i=\deg{\phi_i}_{|V}$ if $V_i:=\phi_i(V)$ has dimension $N-1$ and $d_i=0$ otherwise.
The following proposition expresses the degree of $V$ in terms of the $d_i$. 

This bound is sharp and it  improves by a factor $6N12^{N-1}$ the bound in \cite{viaMD}. If $N=2$ we recover (for curves with CM as well) the same bound as in \cite{EsMordell}.

\begin{propo}\label{prop:bound-altezza}
 Let $V$ as above, then
 \[
  \deg V= 3^{N-1} (N-1)! \sum_{i=1}^N d_i.
 \]
In particular if $H$ is a subgroup of $E^N$ defined by $a_1 X_1+\dotsb a_N X_N=O$, then 
 \[
  \deg H= 3^{N-1} (N-1)! \sum_{i=1}^N \abs{a_i}^2.
 \]
\end{propo}
\begin{proof}
Assume for now that $V_i$ has dimension $N-1$ all $i$.

Let $\ell_j'$ be the pullback in $\P_2^{N-1}$ of $L$ though the projection on the $j$-th coordinate.
Up to a change in the indices, we have that $\ell_j=\phi_i^*(\ell_j')$ for all $j\neq i$.
Let us consider a monomial $\ell_1^{e_1}\dotsm \ell_N^{e_N}$ coming from the expansion of $(\ell_1+\dotsb+\ell_N)^{N-1}$. As the degree is $N-1$ and there are $N$ summands, there must be an index $i$ such that $\ell_i$ does not appear in this monomial. Consider the projection $\phi_i(V)$, which is a copy of $E^{N-1}$ and is equivalent to $(3\ell_1')\dotsm (3\ell_{N-1}')$ in the Chow group of $\P_2^{N-1}$. Up to the renumbering of the $\ell_j'$ we have
\[
 V.(\ell_1^{e_1}\dotsm \ell_N^{e_N})=V.(\phi_i^*(\ell_1'^{e_1})\dotsm \phi_i^*(\ell_{N-1}'^{e_{N-1}}))=\phi_{i*}(V).(\ell_1'^{e_1}\dotsm \ell_{N-1}'^{e_{N-1}})
\]
(see \cite{Fulton}, Example 8.1.7 for the Projection Formula).
Now we show that, when expanding, the factors in which one of the exponents of $\ell_i$ is 2 do not contribute to the degree of $V$.
Indeed $\phi_{i*}(V)=d_i 3^{N-1} (\ell_1'\dotsm\ell_{N-1}')$  and if one of the $e_j$ is at least two then the whole $V.(\ell_1^{e_1}\dotsm \ell_N^{e_N})$ must be zero.
The only monomials of degree $N-1$ with all exponents smaller than two are those with $e_i=0$ and $e_j=1 \forall j\neq i$; these monomials appear with a coefficient of $(N-1)!$ in the expansion of $(\ell_1+\dotsb+\ell_N)^{N-1}$, and for these monomials we have $ V.(\ell_1^{e_1}\dotsm \ell_N^{e_N})=d_i 3^{N-1}$.

In conclusion the degree of $V$ is given by
\[
 \deg V=3^{N-1} (N-1)! \sum_{i=1}^N d_i.
\]

 Notice that this formula remains true if for some of the $i$'s the restriction of $\phi_i$ to $V$ has a smaller dimension. In this case we can omit the indices on which the dimension decreases and apply the same argument.
 
 If $H$ is a subgroup as in the statement, then $d_i$ is the degree of the multiplication by $a_i$, which is $\abs{a_i}^2$ (See \cite{SilvermanAdvancedTopics}, Chapter II, \S 1, Corollary 1.5).
 \end{proof}

\subsection{A bound for the height of a translate}\label{bounddegalt}

Here we  prove some general bounds for the degree and the height of a proper translate $H+P$ in $E^2$ in terms of $\hat h(P)$ and of the coefficients of the equation defining the algebraic subgroup $H$. 

We will often identify algebraic subgroups with matrices and with the lattices generated by their rows as explained in Section~\ref{prelim-S}. %The proofs of this section follow the lines of those in \cite{ExpTAC}, Section 7.2  with some improvements which give a better bound and improve the dependence on $\rend$ that we would obtain from that proof in the CM case.

\subsubsection{A lemma on adjugate matrices}\label{secMatrices}
Let $A$ be a $n\times n$ matrix with complex coefficients. Let $a_i\in\C^n$ be the rows of $A$.
\begin{D}\label{defU*} The \emph{adjugate matrix of $A$}, denoted $A^*$, is  the transpose of the matrix $((-1)^{i+j}\det M_{ij})_{ij}$, where $M_{ij}$ is the $(n-1)\times(n-1)$ minor obtained from $A$ after deleting the $i$-th row and the $j$-th column.
\end{D}
The adjugate matrix has the property that
\[
AA^*=A^*A=(\det A)\mathrm{Id}
\]
and its entries are bounded as it follows:
\begin{lem}\label{lemmaadjugate}
Let $A\in M_{n\times n}(\C)$  be the matrix with rows $a_1,\dotsc,a_n\in\C^n$ and $B\in M_{(n-1) \times n}(\C)$ be the matrix with rows $a_2,\dotsc,a_n$.
Then the norm of the first row of $A^*$ is equal to $\sqrt{\det (B{B}^\dagger)}$ where $B^\dagger$ is the transpose of the complex conjugate.
\end{lem}
\begin{proof}
Applying the Cauchy-Binet formula to the matrices $B,{B}^\dagger$ we have that $\det (B{B}^\dagger)$ is equal to the sum of the squares of the absolute values of the determinants of all $(n-1)\times(n-1)$ minors of $B$, but these determinants are (up to the sign) the entries of the first column of $A^*$.
\end{proof}

\subsubsection{A preliminary bound}\label{subsec-alttrasl}
In order to give a general bound for  the height of $H+P$ we use an argument based on linear algebra, and some bounds on heights from Subsection \ref{SezioneAltezze}. We prove the following proposition, which generalises \cite{EsMordell}, Proposition 5.1 and  improves on \cite{viaMD}, Proposition 4.1:
\begin{propo}\label{boundsotto}Let $P$ be a point in $E^N$. Let $H$ be a component of the  algebraic subgroup in $E^N$ defined by the equation $a_1X_1+a_2 X_2 +\dots +a_N X_N=O$, with $u=(a_1,\dots,a_N)\in \rend^N\setminus\{0\}$. Then 
$$\deg (H+P)\leq 3^{N-1} (N-1)! ||{u}||^2$$ 
where $\norm{u}$ denotes the euclidean norm of $u$,
and 
$$h_2(H+P)\leq  3^{N-1} N!\left(\hat h(u(P))+N C(E)\norm{u}^2 \right)$$
where $u(P)=a_1 P_1+\dotsb + a_N P_N$ and $C(E)$ is defined in Proposition~\ref{confrontoaltezze}.

% and $\gamma_K(n)$ is a generalisation of the Hermite's constant as defined in \cite{ThunderHermiteConstant}.
% In particular (see \cite{VaalerBestConstant}) if $K$ is an imaginary quadratic number field with discriminant $D_K$ we have
% 
% \[
%  \gamma_K(n)^n\leq \left(\frac{2}{\pi}\right)^n \abs{D_K}^\frac{n}{2} n!
% \]
% while if $K=\Q$ then $\gamma_\Q(n)^\frac{1}{2}$ is the usual Hermite's constant, so that 
% $\gamma_\Q(1)=1,\gamma_\Q(2)^2=16/9,\gamma_\Q(3)^3=4,\gamma_\Q(4)^4=16...$ and in general
% \[\gamma_\Q(n)^n\leq \left(\frac{2}{\pi}\right)^{2n}\Gamma\left(2+\frac{n}{2}\right)^4.\]

\end{propo}
\begin{proof}
By \eqref{prop:bound-altezza}  we get
\[
 \deg (H+P)=\deg H \leq 3^{N-1}(N-1)! \norm{u}^2;
\]
this proves the first part of the statement.

Let $\Lambda=\langle u\rangle_{\End(E)}\subseteq\C^N$ be the lattice generated by $u$, and $\Lambda^\perp$ its orthogonal lattice, defined as the set 
\[
 \Lambda^\perp=\{x\in\rend^N\mid \forall\lambda\in\Lambda\quad \scal{x}{\lambda}_\C =0\},
\]
where $\scal{\cdot}{\cdot}_\C$ is the standard hermitian product in $\C^N$.

% 
% By Siegel's Lemma (See \cite{VaalerBestConstant}, Theorem 4) we can find a basis $u_{2},\dotsc,u_N$ of $\Lambda^\perp$ such that 
% \begin{equation}
%  \prod_{i=2}^{N}\norm{u_i}\leq f^{N-1} \gamma_K(N-1)^{\frac{N-1}{2}} \norm{u}.
% \end{equation}

Define $u_1=u$ and let $u_2,\dotsc,u_N$ be the rows of the matrix attached to the abelian subvariety $H^\perp$, so that $\Lambda^\perp=\langle u_2,\dotsc,u_N\rangle_{\End(E)}.$

For any point $P\in E^N$ there are  two points $P_0\in H$, $P^\perp\in H^\perp$, unique up to torsion points in $H\cap H^\perp$, such that $P=P_0+P^\perp$.

Let $U$ be the $N\times N$ matrix with rows $u=u_1,\dotsc,u_N$, and let $\det U$ be its determinant.

Notice that
\[\det \Lambda=\norm{u}\]
and
\begin{equation}\label{orto.lattice}
 \abs{\det U}=\det\Lambda \cdot \det \Lambda^\perp
\end{equation}
because $\Lambda$ and $\Lambda^\perp$ are orthogonal (this follows again from the Cauchy-Binet formula).

We remark that $u(P_0)=0$ because $P_0\in H$, and $u_i(P^\perp)=0$ for all $i=2,\dotsc,N$ because $P^\perp\in H^\perp$.

Therefore

\begin{equation*}
UP^\perp=
\left(
\begin{array}{c}
u(P^\perp)\\
0\\
\vdots\\
0
\end{array}
\right)
=
\left(
\begin{array}{c}
u(P_0+P^\perp)\\
0\\
\vdots\\
0
\end{array}
\right)
=
\left(
\begin{array}{c}
u(P)\\
0\\
\vdots\\
0
\end{array}
\right),
\end{equation*}
hence
\begin{equation*}
(\det U)P^\perp=U^*UP^\perp=U^*\left(
\begin{array}{c}
u(P)\\
0\\
\vdots\\
0
\end{array}
\right)
\end{equation*}
where $U^*$ is the adjugate matrix of $U$ from Definition \ref{defU*}.

Computing canonical heights and applying Lemma \ref{lemmaadjugate} yields
\begin{equation*}
\abs{\det U}^2\hat h(P^\perp)=\hat h\left((\det U)P^\perp\right)=\det(\Lambda^\perp)^2 \hat h(u(P)),
\end{equation*}
so by \eqref{orto.lattice}
\[
   \hat h(P^\perp)=\frac{\hat h(u(P))}{\norm{u}^2}.
\]

Recall inequality \eqref{mu2mu^}, which gives
 \begin{equation*}
\mu_2(H+P)\leq \hat\mu(H+P) +N C(E)
\end{equation*}

{By \cite{preprintPhilippon} we know that}
\[
\hat\mu(H+P)=\hat h(P^\perp)
\]
and therefore, by Zhang's inequality
\begin{align}
\notag h_2(H+P)&\leq N(\deg H)\mu_2(H+P)\leq\\
\notag &\leq N(\deg H)( \hat\mu(H+P) + NC(E)) =\\
\notag &= N(\deg H)(\hat h(P^\perp)+NC(E)) =\\
\label{ultimarigacost}&=  N\deg H \left(\frac{\hat h(u(P))}{\norm{u}^2} +NC(E)\right).
\end{align}

By \eqref{prop:bound-altezza}  we get
\[
 \deg H\leq 3^{N-1}(N-1)! \norm{u}^2,
\]
so \eqref{ultimarigacost} becomes
\[
 h_2(H+P)\leq 3^{N-1} N!\left(\hat h(u(P))+N C(E)\norm{u}^2 \right).\qedhere
\]
\end{proof}

\subsection{Geometry of numbers}\label{geonum}
In this section  we use  classical tools from the Geometry of Numbers that we have to carefully adapt to $\End(E)$-modules endowed with the hermitian product induced by the N\'eron-Tate height. This is essentially a consequence of Minkowski's Second Theorem. More innovative is the use of Minkowski's First Theorem to construct our auxiliary translate with height and degree sharply bounded.

We assume $E$ to have CM and we use the notations of Section~\ref{CM}. 
 
For points $Q,R\in E$, consider the pairing defined by Philippon in \cite{preprintPhilippon}  
\begin{equation}\label{hpairing}\langle Q,R\rangle=\langle Q,R\rangle_{NT}-\frac{1}{\sqrt D} \langle Q,\sqrt D R\rangle_{NT}\end{equation}
where $\langle\cdot,\cdot\rangle_{NT}$ is the N\'eron-Tate pairing and $D$ is a squarefree negative integer such that $K=\qe(\sqrt D)$.
This pairing is hermitian and makes the N\'eron-Tate height into a semi-norm, since $\langle Q, Q\rangle=2\hat h(Q)$. In the following lemma we denote by $\norm{Q}_h^2:=\langle Q, Q\rangle$.

The following Lemma is a refinement and corrected version of \cite{viaMD} Lemma 5.2.
The reader should be mindful of the difference between $\Oo_K$-modules (typically considered in the literature) and $\End(E)$-modules where $\End(E)$ might only  be  an order.
\begin{lem}\label{lemma.good.generators}
   Let $\Gamma\subset E$ be a finitely generated $\End(E)$-module of rank $r$ (in the sense that $\dim_\C \Gamma\otimes_{\End(E)}\C=r$).
   Then there exist elements $g_1\dotsc,g_r\in\Gamma$ which generate a submodule of $\Gamma$ of finite index 
   %such that $$\left[\frac{\Gamma}{\Gamma_{tors}} : \frac{\Lambda}{\Lambda\cap\Gamma_{tors}}\right]\leq \cundici(r)$$
and such that, for all $a_i\in \End(E)$ it holds
 $$\hat{h}\left(\sum_{i=1}^r a_i g_i\right)\geq \cdodici(r,E)\sum_{i=1}^r\abs{a_i}^2\hat{h}(g_i),$$
 where $\cdodici(r,E)=\frac{2^{2r-2}}{r^2(2r)!^2\abs{D_K}^{r}}$.
   
\end{lem}
\begin{proof}
   Let $\Gamma_\C=\Gamma\otimes_{\End(E)}\C$. This is a $\C$-vector space of dimension $r$ and the height function $\hat{h}$ extends to the square of a norm $\norm{\cdot}_h$ on $\Gamma_\C$ through the hermitian product defined by (\ref{hpairing}), as clarified above. Note that  $\Gamma/\Gamma_{tors}$ embeds in $\Gamma_\C$. 
   
   By  \cite{Lang} Ex.13 pag. 168, the $\Oo_K$-module $(\Gamma/\Gamma_{tors}) \otimes \Oo_K$ contains a free submodule $\tilde{\Gamma}$ of finite index with the same rank. Let $p_1\dotsc,p_r$ be an $\Oo_K$-basis of $\tilde{\Gamma}$. We identify $\Gamma_\C$ with $\C^{r}$ through the choice of the basis $p_1,\dotsc,p_r$. $\tilde{\Gamma}$ is a $K$-lattice in $K^{r}$ in the sense of \cite[Definition C.2.1]{BG06}. Let $\lambda_1\dotsc,\lambda_{r}$ be its successive minima.
   
   Applying  Minkowski's Theorem \ref{AdelMink} we get 
   \[
   \omega_{2r}(\lambda_1\dotsm\lambda_{r})^2\leq 2^{r}\abs{D_K}^\frac{r}{2} (\det \tilde\Gamma)^2.
   \]

   Let $v_1\dotsc,v_r$ be elements of $\tilde\Gamma$ which attains the successive minima. Then $\norm{v_i}_h^2=\lambda_i^2$.

   Let $v_i=\sum_{j=1}^r v_{ij}p_j$  with $v_{ij}\in \Oo_K$ and let $\Lambda$ be the $\End(E)$-submodule of $\tilde{\Gamma}$ generated by the $v_i$.

   Now we write $Vol$ for the volumes in the real Lebesgue measure and $Vol\Oo_K=\frac{\abs{D_K}^{1/2}}{2}$ for the volume of $\Oo_K$ in $\C$ (see for instance \cite{Neukirch} Proposition 5.2  where his volume is twice our Lebesgue volume as he says just above the proposition).
   
   Since the $v_i$ form an $\End(E)$-basis of $\Lambda$, the matrix of the $v_{ij}$ has the same determinant of the $\End(E)$-lattice $\Lambda$, thus $\abs{\det(v_i)}=[\tilde{\Gamma}:\Lambda]\det{\tilde\Gamma}$.
   Let $w_i=v_i/\norm{v_i}_h$ for $i=1,\dotsc,r$ and define $B^*$ as
   \[
      B^*=\left\{y\in\C^{r}\,\,{\mathrm{s.t.}}\,\, \norm{\sum_{i=1}^{r}y_i w_i}_h\leq 1\right\},
   \]
   where the coordinates are always expressed in terms of the $p_i$, so that $y=\sum_{i=1}^r y_{i}p_i$.
   
   Since the change of basis $p_i\to w_i$ sends $B^*$ to the unit ball and  $\abs{\det(w_i)}=\frac{\abs{\det(v_i)}}{\prod_{i=1}^{r}\norm{v_i}_h} $; we have that
   \[
   Vol(B^*)=\frac{\omega_{2r}}{\abs{\det(w_i)}^2}=\frac{\omega_{2r}}{[\tilde{\Gamma} : \Lambda]^2(\det\tilde\Gamma)^2}\prod_{i=1}^{r}\norm{v_i}_h^2\leq\frac{2^r\abs{D_K}^{r/2}}{[\tilde{\Gamma} : \Lambda]^2}\leq 2^r\abs{D_K}^{r/2}.
   \]

    Let $e_j,j=1,\dotsc,r$ be an orthonormal basis of $\C^r$ with respect to $\langle \,,\,\rangle$. 
  Let us identify $\mathbb{C}^{r}$ with $\mathbb{R}^{2r}$ and let $e_j$, ${\bf{i}}e_j$ be the real basis in
$\mathbb{R}^{2r}$ (${\bf{i}}$ is the complex  imaginary number such that ${\bf{i}}^2=-1$). Since our pairing is hermitian this is a Real orthonormal basis. 

   Let now $y$ be a fixed point on the boundary of $B^*$, and fix an index $s=1,\dotsc,r$. 
  
Then for each $s$ the set $B^*$ contains the convex closure of the points $\pm |\Re y_s|e_s$
and $\pm{\bf i}e_s$ and  $\pm e_j$,$\pm {\bf{i}} e_j$ for  $j\neq s$. 
For every choice of the signs , these two simplices have in common only the basis and so a set of zero volume, thus  the volume of the union is 
$$\frac{\abs{\Re y_s}2^{2r}}{(2r)!}.$$ 

Similarly, $B^*$ also contains the  convex closure of the points $\pm|\Im y_s|{\bf i}e_s$ and 
$ \pm e_s$ and  $\pm e_j$,$\pm {\bf{i}} e_j$ for  $j\neq s$.
This shows that
\begin{align*}
   \frac{\abs{\Re y_s}2^{2r}}{(2r)!}&\leq Vol(B^*), & \frac{\abs{\Im y_s}2^{2r}}{(2r)!}&\leq Vol(B^*), & s&=1,\dotsc,r
\end{align*}

   Counting real and imaginary part together gives
   \[
   \abs{y_s}\frac{2^{2r}}{(2r)!}\leq (\abs{\Re y_s}+\abs{\Im y_s})\frac{2^{2r}}{(2r)!}\leq 2 Vol(B^*)\quad \forall s=1,\dotsc,r.
   \]
   Summing for all $s$'s we get
   \[
   \sum_{i=1}^{r}\abs{y_i}\leq \frac{r(2r)!}{2^{2r-1}}Vol (B^*)=\frac{r(2r)!\abs{D_K}^{r/2}}{2^{r-1}}\norm{\sum_{i=1}^{r}y_iw_i}_h,
   \]
   because $y$ was taken on the boundary of $B^*$.
   Rewrite now $\frac{y_i}{\norm{v_i}_h}=x_i$ so that  the previous inequality becomes
   \begin{equation}\label{eq.lower.bound.norma}
      \norm{\sum_{i=1}^{r}x_i v_i}_h\geq\frac{2^{r-1}}{r(2r)!\abs{D_K}^{r/2}}\sum_{i=1}^{r}\abs{x_i}\norm{v_i}_h.
\end{equation}

   Choose $p_1',\dotsc,p_r'$ representatives of $p_1,\dotsc,p_r$ in $\tilde\Gamma$ and define
   \[
   g'_i=\sum_{j=1}^r v_{ij}p_j'\quad i=1,\dotsc,r.
   \]
   The $g'_i$ generate a submodule of finite index in $\tilde\Gamma$. In addition, we know that $2\hat{h}\left(\sum_{i=1}^r a_i g'_i\right)=\norm{\sum_{i=1}^r a_i v_i}_h^2$ and $2\hat h (g'_i)=\norm{v_i}_h^2$ and therefore from \eqref{eq.lower.bound.norma} with $x_i=a_i$ it follows that 
   \begin{multline*}
    \hat{h}\left(\sum_{i=1}^r a_i g'_i\right)=\frac{1}{2}\norm{\sum_{i=1}^r a_i v_i}_h^2\geq \frac{1}{2}\frac{2^{2r-2}}{r^2(2r)!^2\abs{D_K}^{r}}\left(\sum_{i=1}^{r}\abs{x_i}\norm{v_i}_h\right)^2\geq\\
    \geq\frac{2^{2r-2}}{r^2(2r)!^2\abs{D_K}^{r}}\sum_{i=1}^{r}\abs{x_i}^2\frac{\norm{v_i}_h^2}{2}=\frac{2^{2r-2}}{r^2(2r)!^2\abs{D_K}^{r}}\sum_{i=1}^{r}\abs{x_i}^2\hat h (g'_i),
   \end{multline*}
    as the square of a sum of positive quantities is bigger than the sum of their squares. Thus the thesis holds for the $\Oo_K$-module generated by the $g'_i$ in $\tilde\Gamma$.
   In order to get a $\End(E)$-submodule of $\Gamma$ simply set $g_i=fg'_i$. Now the $g_i$ lie in $\Gamma$, because $f\Oo_K\subset \End(E)$ and the desired inequality still holds because  multiplicative factors cancel on both sides.
\end{proof}

If $u$ is a vector in $\C^N$, we denote by $\norm{u}$ its euclidean norm and for a linear form $L\in\C[X_1,\dotsc,X_N]$ we denote by $\norm{L}$  the  euclidean norm of the vector of its coefficients.

%Replacing the estimates of  Lemma~\ref{lemma.good.generators} in \cite[Lemma 7.5]{ExpTAC} we obtain the following Lemma.
  We now use Lemma~\ref{lemma.good.generators} to prove a version for CM elliptic curves of \cite[Lemma 7.5]{ExpTAC}, which follows a classical argument.  We repeat the proof for clarity.

\begin{lem}\label{LemmaHabCM}
Let $1\leq m \leq N$  be integers and let $P=(P_1,\dotsc,P_N)\in B\subseteq E^N$, where $B$ is a torsion variety of dimension $m$.

Then, there exist linear forms $L_1,\dotsc,L_m\in\C[X_1,\dotsc,X_N]$ such that $$\norm{L_j}\leq 1\quad \forall j,$$  and for all $\mathbf{t}=(t_1,\ldots,t_N)\in\End(E)^N$
\begin{equation*}
\hat h(t_1P_1+\dotsb+t_NP_N)\leq \cdiciassette(N,m,E) \max_{1\leq j\leq m}\{|L_j(\mathbf{t})|^2\}\hat h (P).
\end{equation*}
The constant $\cdiciassette(N,m,E)$ is given by 
\begin{equation*}
   \cdiciassette(N,m,E)=\frac{mN}{\cdodici(m,E)}=\frac{m^3(2m)!^2\abs{D_K}^{m}N}{2^{2m-2}}.
\end{equation*}
 where $\cdodici(r,E)$ is the constant defined in Lemma~\ref{lemma.good.generators}.
\end{lem}

\begin{proof}
The points $P_i$ lie in a finitely generated subgroup of $E$ of rank $m$.

We apply Lemma~\ref{lemma.good.generators} to this subgroup and we get elements $g_1,\dotsc,g_m\in E$ as in the lemma, which generate an $\End(E)$- submodule of finite index; after multiplying by a well chosen fixed integer $a$ we may also assume that the $P_i$ lie in this submodule of finite index. Then we have that
\begin{equation*}
aP_i= v_{i1}g_1\dotsb+ v_{im}g_m\quad\text{ for }i=1,\dotsc,N\text{ and some }v_{ij}\in \End(E);
\end{equation*}
moreover
\begin{equation}
\label{stellauno}
   \hat h(b_1g_1+\dotsb+b_mg_m)\geq \cdodici(m,E) \sum_{i=1}^{m}(|b_i|^2\hat h(g_i))\quad \forall\mathbf{b}\in \End(E)^m,
\end{equation}
 where $\cdodici(r,E)=\frac{2^{2r-2}}{r^2(2r)!^2\abs{D_K}^{r}}$ is the constant in Lemma~\ref{lemma.good.generators}.

Let $A=\max_{i,j}\{|v_{ij}|^2\hat h(g_j)\}$ and define
\begin{align*}
\tilde{L}_j&=v_{1j}X_1+\dotsb+v_{Nj}X_N &j&=1,\dotsc,m\\
L_j&=\left(\frac{\hat h(g_j)}{NA}\right)^\frac{1}{2}\tilde{L}_j &j&=1,\dotsc,m.
\end{align*}
Notice that we can assume $A>0$, otherwise the point $P$ would be a torsion point, and the thesis of the  lemma would be trivially true. Notice also that $\norm{L_j}\leq 1$.

With these definitions, for every $\mathbf{t}\in\rend^N$ we have that 
\begin{equation*}
a(t_1P_1+\dotsb+t_NP_N)=\sum_{i=1}^m \tilde{L}_j(\mathbf{t})g_j.
\end{equation*}
Therefore
\begin{align}
   \notag \abs{a}^2 \hat h(t_1P_1+\dotsb+t_NP_N)=\hat h\left(\sum_{j=1}^m \tilde{L}_j(\mathbf{t})g_j\right)\leq\sum_{j=1}^m|\tilde{L}_j(\mathbf{t})|^2 \hat h(g_j)=\\
\label{ineqlemma}= NA\sum_{j=1}^m|{L}_j(\mathbf{t})|^2\leq NmA\max_{1\leq j\leq m}\{|{L}_j(\mathbf{t})|^2\}.
\end{align}

If $i_0,j_0$ are the indices for which the maximum is attained in the definition of $A$, then by \eqref{stellauno} we obtain 
\begin{equation*}
   \cdodici(m,E)A=\cdodici(m,E)|v_{i_0 j_0}|^2\hat h(g_{j_0})\leq \hat h(aP_{i_0})\leq \hat h(aP)=|a|^2 \hat h(P).
\end{equation*}
Combining this with inequality \eqref{ineqlemma}, we get the thesis of the lemma.
\end{proof}

Now we prove  the crucial estimate for our application. The central idea is the application of Minkowski's Convex body Theorem: we construct a convex body $\mathcal {S}_T$  and we use Minkowski's theorem to find a non trivial element $u$ of the lattice $\rend^N$ in $\mathcal {S}_T$. This element defines an algebraic subgroup $H$; the bound on the norm of $u$ will ensure that the degree of  $H$ is $\lessapprox T$ (the extra term in the statement appears only for ease of computation and does not play any special role), while the bound on $\abs{L_i(u)}$ will imply that the height of $H+P$  is bounded by $\lessapprox k\hat{h}(P)/T^{\frac{1}{N-1}}$. The parameters $T$ and $k$ will be chosen later suitably in the proof of the main theorem, in order to get the bound on $\hat{h}(P)$.
%This generalises an idea in the simpler dimension $2$ introduced in  Lemma 5.2 of \cite{EsMordell}.

\begin{lem}\label{MinkCM}
   Let $L_i\in\C[X_1,\dotsc,X_N], i=1,\dotsc, N-1$ be $N-1$ independent linear form. With the notations of Section \ref{CM} let $T\geq 1$ and $\kappa\geq\left(\frac{2f\abs{D_K}^{1/2}}{\pi}\right)^\frac{N}{2(N-1)}$ be two real numbers. Then there exist $u\in\rend^N \setminus\{0\}$ such that 
  \begin{align*}
     \norm{u}^2&\leq T^2+\frac{\kappa^2}{T^\frac{2}{N-1}}\sum_{i=1}^{N-1}\norm{L_i}^2\\
     \abs{L_i(u)}&\leq \frac{\kappa \norm{L_i}}{T^\frac{1}{N-1}},
  \end{align*}
  where    $||{u}||$ denotes the euclidean norm of $u$, $||{L}||$ the euclidean norm of the vector of the coefficients of $L$ and $|L(u)|$ is the absolute value of $L(u)$.
\end{lem}

\begin{proof}
      Let $\mathcal {S}_T\subseteq \C^{N}$ be the set of points $(z_1,\dotsc,z_N)$ satisfying the inequalities
 \begin{align*}
    z_1^2+\dotsb + z_N^2&\leq T^2 +\frac{\kappa^2}{T^\frac{2}{N-1}}\sum_{j=1}^{N-1}\norm{L_j}^2\\
    \abs{L_j(z_1\dotsc,z_N)}&\leq \kappa||L_j||/T^\frac{1}{N-1}.
 \end{align*}
   We identify $\C^N$ and $\R^{2N}$ with coordinates $x_j,y_j$, where we write $z_j=x_j+iy_j$ for $j=1,\dotsc,N$.
   As remarked in Section \ref{CM}, the ring $\rend=\Z[\tau]\subseteq \C$, after this identification, has a 2-dimensional volume equal to $f\frac{\abs{D_K}^\frac{1}{2}}{2}$.
   Let also $\mathcal {S}'_T\subseteq \R^{2N}$ correspond to $\mathcal {S}_T$.

 Geometrically, $\mathcal {S}'_T$ is the intersection between a ball and $N-1$ cylinders with $2$-dimensional basis and $2N-2$-dimensional axis.
 
 The statement of the theorem is equivalent to say that $\mathcal {S}_T\cap \rend^N\neq \{0\}$.
 $\mathcal {S}'_T$ is clearly convex and symmetric with respect to the origin, so by Minkowski's Convex Body Theorem if the set $\mathcal {S}'_T$ has a volume bigger than $\left(f\frac{\abs{D_K}^{1/2}}{2}\right)^N 2^{2N}=\left(2f\abs{D_K}^{1/2}\right)^N$, then the intersection  $\mathcal {S}_T\cap \rend^N$ contains points other than the origin.

 Let $L_N(X_1,\dotsc,X_N)\in\C[X_1,\dotsc,X_N]$ be a linear form orthogonal to the $L_j$ and with $\norm{{L_N}}=1$.
 Define $\tilde{x}_j=\Re L_j(z_1,\dotsc,z_N)$ and $\tilde{y}_j=\Im L_j(z_1,\dotsc,z_N)$ for $j=1,\dotsc,N$ . Let $A$ be the matrix of the coordinate change from the $x_j,y_j$ to the $\tilde{x}_j,\tilde{y}_j$.
 
The volume of $\mathcal {S}'_T$ is given by the integral
\[
  \iiint_{\mathcal {S}'_T} 1 \cdot d  x_1 d  y_1\dots d x_{N}d  y_{N}.
\]
In order to compute this integral we perform a change of variable and use the variables $ \tilde x_1, \tilde y_1\cdots,  \tilde x_{N}, \tilde y_{N}$. Using the Fubini-Tonelli theorem, we integrate with respect to $\tilde x_1, \tilde y_1,\dotsc,  \tilde x_{N-1}, \tilde y_{N-1}$ the area of the 2-dimensional circle obtained by intersecting $\mathcal {S}'_T$ with the plane on which the $\tilde x_1, \tilde y_1,\dotsc, \tilde x_{N-1}, \tilde y_{N-1}$ are constant.

After the change of variable, the volume of $\mathcal {S}'_T$ is given  by the integral
\[
\det A \iiint \pi \left(T^2 +\frac{\kappa^2}{T^\frac{2}{N-1}}\sum_{i=1}^{N-1}\norm{L_i}^2-\norm{A(\tilde x_1,\tilde y_1,\dotsc,\tilde x_{N-1},\tilde y_{N-1},0,0)}^2\right) d \tilde x_1 d \tilde y_1\dots d \tilde x_{N-1}d \tilde y_{N-1},
\]
where each pair of variables $\tilde{x}_j,\tilde{y}_j$ for $j=1,\dotsc,N-1$ is jointly integrated over the disk of radius $\frac{\kappa\norm{L_j}}{T^\frac{1}{N-1}}$.
By Hadamard's inequality $\det A\geq (\prod_{j=1}^{N-1}\norm{L_j}^2)^{-1}$ and by the triangular inequality
\[\norm{A(\tilde x_1,\tilde y_1,\dotsc,\tilde x_{N-1},\tilde y_{N-1},0,0)}^2\leq\frac{\kappa^2}{T^\frac{2}{N-1}}\sum_{j=1}^{N-1}\norm{L_j}^2\]
on the integration domain.

Therefore the whole volume can be strictly bounded below as 
\[
 \det A\cdot  \pi T^2\cdot \prod_{j=1}^{N-1}\left(\pi\frac{\kappa^2\norm{L_j}^2}{T^\frac{2}{N-1}} \right)\geq \pi^{N}\kappa^{2(N-1)}
\]
and the hypothesis of Minkowski's theorem are satisfied as soon as
\[
\pi^{N}\kappa^{2(N-1)}\geq (2f\abs{D_K}^{1/2})^N.\qedhere
\]
\end{proof}

\section{The proof of the main Theorem}

We proceed now to the proof of the main theorem of the paper, following the ideas sketched in the Introduction. First  we apply the content of~Section \ref{section:computations} to construct in Theorem~\ref{AusiliareCM} an auxiliary translate with controlled height and degree.
\begin{thm}\label{AusiliareCM}
   Let $E$ be an elliptic curve with CM by $\tau$ and let $K=\Q(\tau)$  with discriminant $D_K$; let $f$ be the conductor of $\End(E)=\Z+\Z\tau$.

   Let $P=(P_1,\dotsc,P_N)\in B\subset  E^N$, where $B$ is a torsion variety of dimension $N-1$.
Let $T\geq 1$ and $\kappa\geq\left(\frac{2f\abs{D_K}^{1/2}}{\pi}\right)^\frac{N}{2(N-1)}$ be real numbers.

Then there exists an abelian subvariety $H\subset  E^N$ of codimension $1$ such that 
\begin{align*}
   \deg(H+P)\leq &  3^{N-1}(N-1)!\left(T+\frac{(N-1)\kappa^2}{T^{\frac{1}{N-1}}}\right)\\
   h_2(H+P)\leq & \cuno(N,E)\frac{\kappa^2}{T^{\frac{1}{N-1}}}\hat{h}(P)+\cdue(N,E)\left(T+\frac{(N-1)\kappa^2}{T^{\frac{1}{N-1}}}\right)
\end{align*}
where
\begin{align*}\cuno(N,E)&=3^{N-1} (N)!\cdiciassette(N,N-1,E)=\frac{N(N-1)3^{N-1} (N)!}{\cdodici(N-1,E)}=\\&=\frac{N(N-1)^3 3^{N-1}(2N-2)!^2 (N)!\abs{D_K}^{N-1}}{2^{2N-4}},\\
\cdue(N,E)&=N3^{N-1}N!C(E)\end{align*}
and $C(E)$ is defined in Proposition \ref{confrontoaltezze}. 
\end{thm}
\begin{proof}
   Let us apply Lemma \ref{LemmaHabCM} to the point $P$ with $m=N-1$. We obtain linear forms $L_1,\dotsc,L_{N-1}\in\C[X_1,\dotsc,X_N]$ with $\norm{L_i}\leq 1$.
   We apply now Lemma \ref{MinkCM} to these linear forms, taking the same $\kappa$ of the statement and a $T$ equal to the square root of the $T$ in the statement. We obtain a non-zero vector $(a_1,\dotsc,a_N)=u\in\End(E)^N$ such that
   \begin{align}
      \label{aux.eqn1}||{u}||^2&\leq T+\frac{\kappa^2(N-1)}{T^\frac{1}{N-1}}\\
      \label{aux.eqn2}\abs{L_i(u)}&\leq \frac{\kappa }{T^\frac{1}{2(N-1)}}.
  \end{align}
  By Lemma \ref{LemmaHabCM}, the bound \eqref{aux.eqn2} implies that 
  \begin{equation}\label{aux.eqn3}
  \hat h(a_1 P_1 +\dotsb a_N P_N)\leq \cdiciassette(N,N-1,E)\frac{\kappa^2\hat h (P)}{T^\frac{1}{(N-1)}}.
   \end{equation}
  
  Let $H$ be the component containing the identity of the algebraic subgroup defined by $a_1 X_1+\dotsb + a_N X_N=0$.
  The thesis now follows from Proposition~\ref{boundsotto}; more precisely
  \begin{align*}
     \deg (H+P)\leq 3^{N-1} (N-1)! \norm{u}^2\leq 3^{N-1}(N-1)!\left(T+\frac{(N-1)\kappa^2}{T^{\frac{1}{N-1}}}\right)
  \end{align*}
from \eqref{aux.eqn1}, and 
  \begin{align*}
     h_2(H+P)&\leq  3^{N-1} N!\left(\hat h(u(P))+N C(E)\norm{u}^2 \right)\leq\\
     &\leq 3^{N-1} N!\left(  \cdiciassette(N,N-1,E)\frac{\kappa^2\hat h (P)}{T^\frac{1}{(N-1)}}+   N C(E)\left( T+\frac{\kappa^2(N-1)}{T^\frac{1}{N-1}}   \right)   \right)
  \end{align*}
from \eqref{aux.eqn1} and \eqref{aux.eqn3}.
\end{proof}
We prove here the theorem with the sharpest constants that come from the method; the theorem as stated in the introduction follows by observing that $D_1\leq C_1+C_3$, $D_2\leq C_2$ and $\Dtre\leq C_4$.

The strategy is the following: the transversality of $\Ci$ ensures that  $P$ is a component of $\Ci \cap (H+P)$; we then apply the Arithmetic B\'ezout Theorem to $\Ci \cap (H+P)$; our sharp bounds for  height and degree  of $H+P$ and a good choice of the parameters give the desired bound for $h(P)$.
\begin{thm}\label{MAINTCM}
Let $E$ be an elliptic curve with CM. Let $\Ci$ be an irreducible transverse curve in $E^N$.
Then every point $P$ on $\Ci$ of rank  $\le N-1$  has height bounded as:
   \begin{equation*}
      h_2(P)\leq \Cuno(N,E)\cdot h_2(\Ci)(\deg\Ci)^{N-1} +\Cdue(N,E)(\deg\Ci)^N +\Cquattro(N,E)h_2(\Ci)+\Ctre(N,E),
\end{equation*}
where
\begin{align*}
   \Cuno(N,E)&=N!\left(\frac{N}{N-1}\right)^{N-1}3^{N-1}\cuno(N,E)^{N-1}\left(\frac{2f\abs{D_K}^{1/2}}{\pi}\right)^N,\\
   \Cdue(N,E)&=N!\left(\frac{N}{N-1}\right)^{N-1}3^{N-1}\cuno(N,E)^{N-1}\left(\frac{2f\abs{D_K}^{1/2}}{\pi}\right)^N\left(N^2 \cquattro(E)+\csei(1,N-1,3^{N}-1)\right),\\
   \Cquattro(N,E)&=\frac{4^{N-1}(2N-1)^2}{(N-1)(2N)!^2\abs{D_K}^{N-1}},\\
   \Ctre(N,E)&=\frac{4^{N-1}(2N-1)^2}{(N-1)(2N)!^2\abs{D_K}^{N-1}}\left(N^2\cquattro(E)+\csei(1,N-1,3^N-1)\right)+N\cquattro(E),
 \end{align*}
 where $\cuno(N,E)$ is defined in Proposition \ref{AusiliareCM}, $\cquattro(E)$ in Proposition \ref{confrontoaltezze} and $\csei(1,N-1,3^{N}-1)$ in Theorem \ref{AriBez}.
\end{thm}
\begin{proof}
 If $P$ has rank zero then its height is zero and the statement is true.

 Let $\kappa=\left(\frac{2f\abs{D_K}^{1/2}}{\pi}\right)^\frac{N}{2(N-1)}$, and let $T\geq 1$ be real numbers whose value will be chosen later.
 We apply Proposition~\ref{AusiliareCM} to the point $P$ of rank one, thus obtaining a subvariety with
\begin{align}\label{bounds}
   \deg(H+P)\leq &  3^{N-1}(N-1)!\left(T+\frac{(N-1)\kappa^2}{T^{\frac{1}{N-1}}}\right)\\
   h_2(H+P)\leq & \cuno(N,E)\frac{\kappa^2}{T^{\frac{1}{N-1}}}\hat{h}(P)+\cdue(N,E)\left(T+\frac{(N-1)\kappa^2}{T^{\frac{1}{N-1}}}\right)
\end{align}

We now want to bound $\hat h(P)$  in terms of $\deg (H+P)$ and $h_2(H+P)$.

Notice that the point $P$ is a component of the intersection $\Ci\cap (H+P)$, because otherwise $\Ci\subseteq H+P$,
contradicting the fact that $\Ci$ is transverse. Therefore we can apply the Arithmetic B\'ezout Theorem \ref{AriBez} to the intersection
$\Ci\cap (H+P)$, obtaining:
\begin{equation*}\label{boundh2}
h_2(P)\leq h_2(\Ci)\deg H+ h_2(H+P)\deg \Ci+\csei(1,N-1,3^N-1)\deg H\deg \Ci.
\end{equation*}

By Proposition \ref{confrontoaltezze} we have $\hat h(P)\leq h_2(P)+N\cquattro(E)$
 so, using the bounds in formula \eqref{bounds}, we get
\begin{align}\label{eq_ottim}
 h_2(P)\leq \frac{\csette \kappa^2}{T^{\frac{1}{N-1}}}h_2(P)+\cotto T+\frac{\cnove \kappa^2}{T^{\frac{1}{N-1}}}
\end{align}

\bigskip

with
\begin{align*}
 \csette(\Ci,N,E)&=\cuno(N,E)\deg\Ci,\\
 \cotto(\Ci,N,E)&=3^{N-1}(N-1)!h_2(\Ci)+\cdue(N,E)\deg\Ci+\csei(1,N-1,3^N-1) 3^{N-1}(N-1)!\deg\Ci, \\
 \cnove(\Ci,N,E)&=(N-1)\cotto(\Ci,N,E)+N\cquattro(E)\csette(\Ci,N,E),
\end{align*}
% so that
% \begin{equation}\label{eq_ottim}
% \hat h(P)\leq \cundici\frac{\kappa^2}{T}\hat h(P)+\cotto T+\ctre.
% \end{equation}
We set
\begin{equation*}
   T=\left(\frac{N}{N-1}\csette\kappa^2\right)^{N-1},
\end{equation*}
so that
the coefficient in front of $h_2(P)$ on the right-hand side of \eqref{eq_ottim} becomes $\frac{N-1}{N}$. Bringing it to the other side we obtain
\begin{align}
   h_2(P)\leq N\left(\cotto T+\frac{\cnove \kappa^2}{T^{\frac{1}{N-1}}}\right)=N\left(\frac{N}{N-1}\right)^{N-1}\cotto\csette^{N-1}\kappa^{2(N-1)}+(N-1)\frac{\cnove}{\csette}.
\end{align}

The ratio $\frac{\cnove}{\csette}$ can be estimated as

\begin{multline*}
   \frac{\cnove}{\csette}\leq \frac{(N-1)\cotto}{\cuno(N,E)\deg\Ci}+N\cquattro(E)\leq\\
\leq\frac{N-1}{\cuno(N,E)}\left( 3^{N-1}(N-1)!h_2(\Ci)+\cdue(N,E)+\csei(1,N-1,3^N-1) 3^{N-1}(N-1)!\right)+N\cquattro(E).
\end{multline*}
So that after substituting the values of $\kappa,\csette,\cotto$ into \eqref{eq_ottim} obtain the thesis.
This can be readily checked by computing separately the coefficients of the terms in $h_2(\Ci)(\deg\Ci)^{N-1}$,$(\deg\Ci)^N$,$h_2(\Ci)$.
Indeed the coefficient of the term in $h_2(\Ci)(\deg\Ci)^{N-1}$ is 
\[N\left(\frac{N}{N-1}\right)^{N-1}3^{N-1}(N-1)!\cuno(N,E)^{N-1}\kappa^{2(N-1)}.\]
The coefficient of $(\deg\Ci)^{N}$ is 
\[N\left(\frac{N}{N-1}\right)^{N-1}\left(\cdue(N,E)+\csei(1,N-1,3^N-1) 3^{N-1}(N-1)!\right)\cuno(N,E)^{N-1}\kappa^{2(N-1)}.\]
The coefficient of $h_2(\Ci)$ is just 
\[\frac{(N-1)^2 3^{N-1}(N-1)!}{\cuno(N,E)}.\]
\end{proof}

%This theorem can be combined with the upper bounds proved in the main theorem in order to compute all the rational points on certain infinite families of curves.

\begin{remark}
The exponents $N$ and $N-1$ in the dependence from $\deg \Ci$ are related to the rank of the points, rather than to the dimension of the ambient space. If we restrict ourselves to points of rank $\leq r$ for an $r<N-1$ we can project the curve $\Ci$ to $E^{r+1}$ in a way that doesn't increase its height and degree, and apply our theorem to the image obtaining a bound with better exponents.
The whole procedure is described in detail in the proof of~\cite[Theorem 4.3]{EsMordell}.
\end{remark}

\section{Elliptic curves without Complex Multiplication}
The case of an elliptic curve $E$ without Complex Multiplication was  treated with a different method in \cite{ExpTAC}. The methods that we have presented so far in this paper allow us to improve drastically the constants, compared to those in \cite{ExpTAC}. We collect in this section the propositions that need to be modified. 

The following is a version over the reals of Lemma~\ref{MinkCM}.

\begin{lem}\label{MinknonCM}
   Let $L_i\in\R[X_1,\dotsc,X_N], i=1,\dotsc, N-1$ be $N-1$ independent linear form. With the notations of Section \ref{CM} let $T\geq 1$ and $\kappa\geq 2^\frac{N}{N-1}$ be two real numbers. Then there exist $u\in\Z^N \setminus\{0\}$ such that 
  \begin{align*}
     \norm{u}^2&\leq T^2+\frac{\kappa^2}{T^\frac{2}{N-1}}\sum_{i=1}^{N-1}\norm{L_i}^2\\
     \abs{L_i(u)}&\leq \frac{\kappa \norm{L_i}}{T^\frac{1}{N-1}},
  \end{align*}
  where    $||{u}||$ denotes the euclidean norm of $u$, $||{L}||$ the euclidean norm of the vector of the coefficients of $L$ and $|L(u)|$ is the absolute value of $L(u)$.
\end{lem}

\begin{proof}
      Let $\mathcal {S}_T\subseteq \R^{N}$ be the set of points $(x_1,\dotsc,x_N)$ satisfying the inequalities
 \begin{align*}
    x_1^2+\dotsb + x_N^2&\leq T^2 +\frac{\kappa^2}{T^\frac{2}{N-1}}\sum_{j=1}^{N-1}\norm{L_j}^2\\
    \abs{L_j(x_1\dotsc,x_N)}&\leq \kappa||L_j||/T^\frac{1}{N-1}.
 \end{align*}

 The statement of the theorem is equivalent to say that $\mathcal {S}_T\cap \Z^N\neq \{0\}$.
 $\mathcal {S}_T$ is clearly convex and symmetric with respect to the origin, so by Minkowski's Convex Body Theorem if the set $\mathcal{S}_T$ has a volume bigger than $2^{N}$, then the intersection  $\mathcal{S}_T\cap \Z^N$ contains points other than the origin.

 Let $L_N(X_1,\dotsc,X_N)\in\R[X_1,\dotsc,X_N]$ be a linear form orthogonal to the $L_j$ and with $\norm{{L_N}}=1$.
 Define $\tilde{x}_j=L_j(x_1,\dotsc,x_N)$ for $j=1,\dotsc,N$ . Let $A$ be the matrix of the coordinate change from the $x_j$ to the $\tilde{x}_j$.
 
The volume of $\mathcal {S}_T$ is given by the integral
\[
  \iiint_{\mathcal {S}_T} 1 \cdot d  x_1 \dots d x_{N}.
\]
In order to compute this integral we perform a change of variable and use the variables $ \tilde x_1,\cdots,  \tilde x_{N}$. Using the Fubini-Tonelli theorem, we integrate with respect to $\tilde x_1,\dotsc,  \tilde x_{N-1}$ the length area of the segment obtained by intersecting $\mathcal {S}_T$ with the plane on which the $\tilde x_1,\dotsc, \tilde x_{N-1}$ are constant.

After the change of variable, the volume of $\mathcal {S}_T$ is given  by the integral
\[
\det A \iiint \left(T^2 +\frac{\kappa^2}{T^\frac{2}{N-1}}\sum_{i=1}^{N-1}\norm{L_i}^2-\norm{A(\tilde x_1,\dotsc,\tilde x_{N-1},0)}^2\right)^{1/2} d \tilde x_1 \dots d \tilde x_{N-1},
\]
where each variable $\tilde{x}_j$ for $j=1,\dotsc,N-1$ is integrated from $-\frac{\kappa\norm{L_j}}{T^\frac{1}{N-1}}$ to $\frac{\kappa\norm{L_j}}{T^\frac{1}{N-1}}$.
By Hadamard's inequality $\det A\geq (\prod_{j=1}^{N-1}\norm{L_j})^{-1}$ and by the triangular inequality
\[\norm{A(\tilde x_1,\dotsc,\tilde x_{N-1},0)}^2\leq\frac{\kappa^2}{T^\frac{2}{N-1}}\sum_{j=1}^{N-1}\norm{L_j}^2\]
on the integration domain.

Therefore the whole volume can be strictly bounded below as 
\[
 \det A\cdot  T\cdot \prod_{j=1}^{N-1}\left(\frac{\kappa\norm{L_j}}{T^\frac{1}{N-1}} \right)\geq \kappa^{N-1}
\]
and the hypothesis of Minkowski's theorem are satisfied as soon as
\[
\kappa^{N-1}\geq 2^N.\qedhere
\]
\end{proof}

The following is a version without Complex Multiplication of Theorem~\ref{AusiliareCM}

\begin{thm}\label{AusiliarenonCM}
   Let $E$ be an elliptic curve without CM.

   Let $P=(P_1,\dotsc,P_N)\in B\subset  E^N$, where $B$ is a torsion variety of dimension $N-1$.
Let $T\geq 1$ and $\kappa\geq 2^\frac{N}{N-1}$ be real numbers.

Then there exists an abelian subvariety $H\subset  E^N$ of codimension $1$ such that 
\begin{align*}
   \deg(H+P)\leq &  3^{N-1}(N-1)!\left(T+\frac{(N-1)\kappa^2}{T^{\frac{1}{N-1}}}\right)\\
   h_2(H+P)\leq & \cquindici(N)\frac{\kappa^2}{T^{\frac{1}{N-1}}}\hat{h}(P)+\cdue(N,E)\left(T+\frac{(N-1)\kappa^2}{T^{\frac{1}{N-1}}}\right)
\end{align*}
where
\[\cquindici(N)=\frac{N(N-1)^3 3^{N-1} N! (N-1)!^4}{4^{N-2}},\]
\[\cdue(N,E)=N3^{N-1}N!C(E)\]
and $C(E)$ is defined in Proposition \ref{confrontoaltezze}.
\end{thm}
\begin{proof}
   Let us apply \cite[Lemma 7.5]{ExpTAC} to the point $P$ with $m=N-1$. We obtain linear forms $L_1,\dotsc,L_{N-1}\in\R[X_1,\dotsc,X_N]$ with $\norm{L_i}\leq 1$ and such that
   \begin{equation*}
\hat h(t_1P_1+\dotsb+t_NP_N)\leq \cquattordici(N) \max_{1\leq j\leq m}\{|L_j(\mathbf{t})|^2\}\hat h (P).
\end{equation*}
for all $\mathbf{t}\in\Z^N$, where
\[
\cquattordici(N)=\frac{N(N-1)^3(N-1)!^4}{4^{N-2}}.
\]

   We apply now Lemma \ref{MinknonCM} to these linear forms, taking the same $\kappa$ of the statement and a $T$ equal to the square root of the $T$ in the statement. We obtain a non-zero vector $(a_1,\dotsc,a_N)=u\in\Z^N$ such that
   \begin{align}
      \label{aux.eqn1.nonCM}||{u}||^2&\leq T+\frac{\kappa^2(N-1)}{T^\frac{1}{N-1}}\\
      \label{aux.eqn2.nonCM}\abs{L_i(u)}&\leq \frac{\kappa }{T^\frac{1}{2(N-1)}}.
  \end{align}
  Therefore by the definition of the $L_j$ the bound \eqref{aux.eqn2.nonCM} implies that 
  \begin{equation}\label{aux.eqn3.nonCM}
  \hat h(a_1 P_1 +\dotsb a_N P_N)\leq \cquattordici(N)\frac{\kappa^2\hat h (P)}{T^\frac{1}{(N-1)}}.
   \end{equation}
  
  Let $H$ be the component containing the identity of the algebraic subgroup defined by $a_1 X_1+\dotsb + a_N X_N=0$.
  The thesis now follows from Proposition~\ref{boundsotto}; more precisely
  \begin{align*}
     \deg (H+P)\leq 3^{N-1} (N-1)! \norm{u}^2\leq 3^{N-1}(N-1)!\left(T+\frac{(N-1)\kappa^2}{T^{\frac{1}{N-1}}}\right)
  \end{align*}
from \eqref{aux.eqn1.nonCM}, and 
  \begin{align*}
     h_2(H+P)&\leq  3^{N-1} N!\left(\hat h(u(P))+N C(E)\norm{u}^2 \right)\leq\\
     &\leq 3^{N-1} N!\left(  \cquattordici(N)\frac{\kappa^2\hat h (P)}{T^\frac{1}{(N-1)}}+   N C(E)\left( T+\frac{\kappa^2(N-1)}{T^\frac{1}{N-1}}   \right)   \right)
  \end{align*}
from \eqref{aux.eqn1.nonCM} and \eqref{aux.eqn3.nonCM}.
\end{proof}

With this version of Theorem~\ref{AusiliareCM} the proof of Theorem~\ref{MAINTCM} can be replicated closely. One needs only to take $\kappa=2^\frac{N}{N-1}$ instead of $\kappa=\left(\frac{2f\abs{D_K}^{1/2}}{\pi}\right)^\frac{N}{2(N-1)}$, and to replace $\cuno(N,E)$ from Theorem~\ref{AusiliareCM} by $\cquindici(N)$ from Theorem~\ref{AusiliarenonCM}. Carrying out the computations one obtains the following:

\begin{thm}\label{MAINTnonCM}
Let $E$ be an elliptic curve without CM. Let $\Ci$ be an irreducible transverse curve in $E^N$.
Then every point $P$ on $\Ci$ of rank  $\le N-1$  has height bounded as:
   \begin{equation*}
      h_2(P)\leq \Ccinque(N,E)\cdot h_2(\Ci)(\deg\Ci)^{N-1} +\Csei(N,E)(\deg\Ci)^N +\Csette(N,E)h_2(\Ci)+\Cotto(N,E),
\end{equation*}
where
\begin{align*}
   \Ccinque(N,E)&=N!\left(\frac{N}{N-1}\right)^{N-1}3^{N-1}4^N\cquindici(N)^{N-1},\\
   \Csei(N,E)&=N!\left(\frac{N}{N-1}\right)^{N-1}3^{N-1}4^N\cquindici(N)^{N-1}\left(N^2 \cquattro(E)+\csei(1,N-1,3^{N}-1)\right),\\
   \Csette(N,E)&=\frac{4^{N-1}}{(N-1)(N)!^2(N-1)!^2},\\
   \Cotto(N,E)&=\frac{4^{N-1}}{(N-1)(N)!^2(N-1)!^2}\left(N^2\cquattro(E)+\csei(1,N-1,3^N-1)\right)+N\cquattro(E),
 \end{align*}
 where $\cquindici(N)$ is defined in Proposition \ref{AusiliarenonCM}, $\cquattro(E)$ in Proposition \ref{confrontoaltezze} and $\csei(1,N-1,3^{N}-1)$ in Theorem \ref{AriBez}.
\end{thm}

By a rougher estimation of the constants this theorem gives, in turn, Theorem~\ref{thm:insieme}, part \ref{caso_E^N-non-CM} of the Introduction.

\section{A lower bound for non-integral points}
We present here a version for the points of $\Ci(K)$ of the result given for $\Ci(\Q)$ by M. Stoll in the Appendix to \cite{EsMordell}. The arguments of that Appendix go through with little modification as long as the prime $\ell$ is assumed to split completely in the field $K$. Notice that there are infinitely many of such primes.

Let $E$ be an elliptic curve over~$K$ of rank~$1$ given by a
Weierstrass equation with integral coefficients.
In this section, we consider a curve $\Ci \subseteq E \times E$ that is given
by an affine equation of the form
\[ F_1(x_1,y_1) = F_2(x_2,y_2) \]
(where $(x_1,y_1)$ are the affine coordinates on the first
and $(x_2,y_2)$ on the second factor~$E$)
with polynomials $F_1, F_2 \in \Oo_K[x,y]$. Using the equation of~$E$,
we can assume that $F_j(x,y) = f_j(x) + g_j(x) y$ with univariate
polynomials $f_j, g_j \in \Oo_K[x]$. Note that $F_j$ is a rational function
on~$E$ whose only pole is at the origin~$O$ and that
$d_j := \deg F_j = \max\{2 \deg f_j, 3 + 2 \deg g_j\}$.
The \emph{leading coefficient} of~$F_j$ is the coefficient of the
term of largest degree present in~$F_j$.
We also require in the following that $d_1$ is strictly greater than~$d_2$.
Our goal in this section is to obtain a \emph{lower} bound
on the height of a point $P \in \Ci(K)$.

Let $\ell$ be an odd prime number that splits completely in $K$, and fix an extension of the $\ell$-adic absolute value to $K$. The completion of $K$ with respect to this absolute value is $\Q_\ell$ just as discussed in the Appendix to \cite{EsMordell}.

Let $S$ be a finite set of primes containing the primes dividing
the leading coefficients of $F_1$ and~$F_2$ and also the primes above $2$
if the equation defining~$E$ contains mixed terms.
We denote the ring of $S$-integers by~$\Oo_S$.

Then we obtain the following version of Theorem A.3 of \cite{EsMordell}

\begin{thm} \label{T:lowerbound}
  Consider $E$, $\Ci$ and~$S$ as above (with $d_1 > d_2$). Set
  \[ \lambda = \hat{h}(P_0) \min \{a_\ell^2 \ell^{2\lceil d_1/d_2 \rceil - 2} : \ell \notin S\text{ and $\ell$ splits completely in $K$}\}, \]
  where $P_0$ is a generator of the free part of~$E(K)$ and
  $a_\ell$ is the smallest positive integer such that $a_\ell P_0$ reduces to infinity modulo $\ell$.
  Then
  \[ \Ci(K) \subseteq \{(O,O)\} \cup \bigl(E(\Oo_S) \times E(\Oo_S)\bigr)
                     \cup \{P \in E(K) \times E(K) : \hat{h}(P) \ge \lambda\}.
  \]
\end{thm}

\section{Proof of Theorem~\ref{teorema:esempio}}\label{example}
Let $E$ be the elliptic curve
\[
 E:\, y^2=x^3+2.
\]
This curve $E$ has complex multiplication by a third root of $1$, given by $(x,y)\mapsto (\zeta x,y)$, with $\zeta=\frac{-1+\sqrt{-3}}{2}$. Its discriminant is 
1720 and its $j$-invariant is 0. The set $E(\Q)$ is isomorphic to $\Z$, with generator $(-1,1)$. The canonical height of the generator is $ \hat{h}(g)\approx 1.1319$.

The field of complex multiplication is $K=\Q(\zeta)$, which has discriminant $D_K=-3$. The ring of integers is $\Z[\zeta]$, so $f=1$. 

The set $E(K)$ is an $\rend$-module of rank (as an $\rend$-module) equal to the rank of $E(\Q)$ as an abelian group.

The constant $\cquattro(E)$ is bounded as $\cquattro(E)\leq 6.211$

$\cuno(2,E)=144$

\begin{align*}
 \csei(1,1,8)&\leq 6.019\\
 \cquattro(E)&\leq 6.211\\
 \cuno(2,E)&=144\\
 \Cuno(2,E)&\leq 2101\\
 \Cdue(2,E)&\leq 67638\\
 \Cquattro(2,E)&\leq0.03\\
 \Ctre(2,E)&\leq13.1
\end{align*}

By standard arguments of algebraic geometry and the theory of heights (see~\cite{EsMordell}, Corollary 7.1 and Proposition 8.1 for the details) we have that
\begin{align*}
 \deg \Ci_n&=6n+9,\\
 h_2(\Ci_n)&\leq 6\log 5 (2n+3),\\
 \deg \Di_n&=6\phi(n)+9,\\
 h_2(\Di_n)&\leq 6(2\phi(n)+3)\left(\phi(n)\log 2 +\log 5\right).
\end{align*}

From our main theorem it follows that
\begin{align}
 \notag h_2(P)&\leq 644391 \cdot (2n+3)^2+14\\
 \label{upperBexample}\hat{h}(P)&\leq 644391 \cdot (2n+3)^2+28
\end{align}
if $P\in\Ci_n$,
while
\begin{align}
 \notag h_2(P)&\leq 644391 \cdot (2\phi(n)+3)^2+14\\
 \label{upperBexampleD}\hat{h}(P)&\leq 644391 \cdot (2\phi(n)+3)^2+28
\end{align}
if  $P\in\Di_n$.

Writing
$P=(P_1,P_2)=([a]g,[b]g)$ where  $a$ and $b$ in $\Z[\zeta]$ and $g=(-1,1)$ is a generator of $E(\Q)$, we can easily bound $a,b$ in terms of the height of the point.
Indeed we have that $\hat{h}(P)=\hat h([a]g)+\hat h([b]g)=(\abs{a}^2+\abs{b}^2)\hat{h}(g)$, and for a $P\in\Ci_n$ we have that $h_W(y(P_2))= n h_W(x(P_1))$, so that
\[
  n\abs{a}^2\hat{h}(g)\ll n h_W(x([a]g)) = h_W(y([b]g))\leq h_W([b]g)\ll \hat h([b]g)\ll \abs{b}^2\hat{h}(g)
\]
from which we obtain $n \abs{a}^2\ll \hat{h}(P)$ (and analogously for $\Di_n$).

The details with the exact computation of the constants can be found in \cite{EsMordell}, Theorem 7.3, which gives

%From \cite{EsMordell}, Theorem 7.3 (the bounds on $a,b$ in terms of $\hat{h}(P)$ do not depend on whether the curve has complex multiplication) follows that
\begin{equation}\label{esempio:bound:a}\abs{a}\leq \left(\frac{3 h_2(P)+10.15 n+6}{(2n+3)\hat h(g)}\right)^{1/2}\leq 1307 (2n+3)^{1/2}\end{equation}
and
\begin{equation}\label{esempio:bound:b}|b|\leq \left(\frac{2n h_2 (P)+7.71 n+18}{(2n+3)\hat h(g)}\right)^{1/2}.\end{equation}

For the application of Theorem \ref{T:lowerbound} we can take $S=\emptyset$ and $\ell=2$ (because $-3\equiv 1\pmod 4$); $d_1$ and $d_2$ are $2n$ and $3$ respectively, and $a_2=2$ because $[2](-1,1)=(\frac{17}{4},-\frac{71}{8})$. The lower bound on the the height for a point on $\Ci(K)$ then is given by $\lambda=4^{2n/3 +1}$.
Comparing this lower bound with the upper bound \eqref{upperBexample} and using Theorem \ref{T:lowerbound} we see that for every $n>21$ we have that $\Ci_n(K)=\Ci_n(\Oo_K)$ and only the curves with $n\leq 20$ need to be checked for additional  points.

It can be checked that $E(\Oo_K)=\{\pm g, \pm \zeta g, \pm \zeta^2 g \}$, from which we obtain the points listed in the statement of Theorem~\ref{teorema:esempio} (for the values of cyclotomic polynomials evaluated at sixth roots of unity, see \cite{cycloroots}). A computer calculation (detailed in the next subsection) shows that there are no other points on $\Ci_n(K)$ for $n\leq 20$, which completes the proof of the theorem.

\subsection{Description of the algorithm}
This is the algorithm that we used to perform the computation. It is a modified version of the one described in \cite{EsMordell} and runs in \cite{PARI}. We thank Bill Allombert for his help with PARI and the coding of the algorithm.

\begin{algorithm}[H]
  \caption{Checking for rational points on $\Ci_n(K)$}
  \label{Algo}
  \begin{algorithmic}[1]
      \STATE $E:=$ the elliptic curve defined by $y^2=x^3+2$
      \STATE $\texttt{g}:=$ the point $(-1,1)\in E(\Q)$
      \STATE $\zeta$ a primitive third root of unity
      \FOR{$\texttt{n}=1$ \TO $20$}
	 \STATE $\texttt{Ma}:=$ the upper bound for $\abs{a}$ in equation \eqref{esempio:bound:a}
	 \STATE $\texttt{Mb}:=$ the upper bound for $\abs{b}$ in equation \eqref{esempio:bound:b}
	 \STATE $\texttt{p}:=7$
	 \STATE Initialise $\texttt{L}$ to a list containing all pairs of integers $(a,b)$ with $\abs{a}\leq \texttt{Ma}$ and $1\leq b\leq \texttt{Mb}$
	 \STATE $\texttt{c}:=0$
	 \WHILE{\TRUE}
	    \STATE $\texttt{g2}:=$ the reduction modulo $\texttt{p}$ of the point $\zeta \texttt{g}$
	    \STATE $\texttt{NL}:=$ the cardinality of $\texttt{L}$
	    \STATE $E_p:=$ the reduction of $E$ modulo $\texttt{p}$
	    \STATE $\texttt{Np}:=$ the cardinality of $E_p$
	    \STATE $\texttt{Mpa}:=\min(\texttt{Ma},\texttt{Np}-1)$
	    \STATE $\texttt{Mpb}:=\min(\texttt{Mb},\texttt{Np}-1)$
	    \FOR{$\texttt{a}=-\texttt{Mpa}$ \TO $\texttt{Mpa}$}
	       \STATE $\texttt{ag}:=$ the point $[\texttt{a}]\texttt{g}\in E_p$
	       \FOR{$\texttt{b}=1$ \TO $\texttt{Mpb}$}
		  \STATE $\texttt{ag}:=$ the point $\texttt{ag}+\texttt{g2}$ in $E_p$
		  \IF {$\texttt{ag}$ is the point at infinity}
		     \STATE Remove the pair $(\texttt{a},\texttt{b})$ from $\texttt{L}$
		     \STATE \textbf{next}
		  \ENDIF
		  \STATE $\texttt{x}:=$ the first coordinate of the point $\texttt{ag}$
		  \IF {The congruence $X^3+2\equiv \texttt{x}^{\texttt{2*n}} \pmod{\texttt{p}}$ has no solution}
		  \STATE Remove from $\texttt{L}$ all pairs $(a,b)$ such that $a\equiv \texttt{a}\pmod{\texttt{Np}}$ and  $b\equiv \texttt{b}\pmod{\texttt{Np}}$
		  \ENDIF
	       \ENDFOR
	    \ENDFOR
	    \IF {The cardinality of $\texttt{L}$ is equal to $\texttt{c}$}
	       \STATE $\texttt{c}:=\texttt{c}+1$
	    \ENDIF
	    \IF {The cardinality of $\texttt{L}$ is zero, \OR $\texttt{c}>15$}
	       \STATE \textbf{break}
	    \ENDIF
	    \STATE $\texttt{p}:=$ the next prime after $\texttt{p}$ which is congruent to 1 modulo 3
	 \ENDWHILE
      \ENDFOR
   \end{algorithmic}
 \end{algorithm}

The core of the algorithm is the \texttt{while} loop in line 10. This loop iterates over the prime \texttt{p}, which is always chosen to be congruent to 1 modulo 3. At each iteration the algorithm takes the list \texttt{L}, which initially contains all pairs $(a,b)$ satisfying the bounds derived from our main theorem, and checks for which of these values there exist a point $([a+b \zeta]g,P)$ on the curve $\Ci_n$ reduced modulo \texttt{p}. 
In order to perform this check, for each prime  \texttt{p} we compute the reduction modulo \texttt{p} of the point $[\zeta]g=(-\zeta,1)$ in $E(\mathbb{F}_p)$. Then in the \texttt{for} loops at lines 17 and 19 we iterate on the real and imaginary part of $\texttt{a}+\texttt{b}\zeta$ . The pairs $(\texttt{a},\texttt{b})$ that do not correspond to points modulo \texttt{p} are removed from the list \texttt{L} at line 27. The algorithm then changes the prime number \texttt{p} to the next one still congruent to 1 modulo 3, and the loop starts again.
The algorithm keeps sieving through the list \texttt{L} until either the list become empty, or 15 iterations pass without any new pair being discarded from \texttt{L}. When this happens the program outputs the values remaining in the list \texttt{L}, which are candidate solutions and need to be investigated further.

We need to select only primes congruent to 1 modulo 3, otherwise the reduction of the point $[\zeta]g$ would not be defined over $\mathbb{F}_p$.

For the curve $\Di_n$ it is enough to replace $x^{2n}$ with $\Phi_n(x)^2$ in line 26.

\section*{Acknowledgement}
We warmly thank Bill Allombert for suggesting improvements to Algorithm \ref{Algo} and writing a very fast PARI/GP implementation, suitable to our applications.
We also thank John Cremona and Angelos Koutsianas for computing for us the points of $E(\Oo_K)$.

\def\cprime{$'$}
\providecommand{\bysame}{\leavevmode\hbox to3em{\hrulefill}\thinspace}
\providecommand{\MR}{\relax\ifhmode\unskip\space\fi MR }
\providecommand{\MRhref}[2]{%
  \href{http://www.ams.org/mathscinet-getitem?mr=#1}{#2}
}
\providecommand{\href}[2]{#2}

\bigskip

\noindent \textsc{Francesco Veneziano}:\\
Dipartimento di Matematica, Universit{\`a} degli Studi di Genova,\\
Via Dodecaneso 35,\\
16146, Genova (GE), Italia.\\
email: \href{mailto:veneziano@dima.unige.it}{veneziano@dima.unige.it}																																																															

\medskip

\noindent \textsc{Evelina Viada}:\\
Mathematisches Institut,
Georg-August-Universit\"at,\\
Bunsenstra\ss e 3-5,\\
D-37073, G\"ottingen,
Germany.% and
%ETH Z\"urich,
%R\"amistrasse 101,
%8092, Zurich,
%Switzerland.

\noindent email: \href{mailto:evelina.viada@math.ethz.ch}{evelina.viada@math.ethz.ch}

\end{document}